\documentclass[11pt]{amsart}

\usepackage{amsaddr}

\usepackage{hyperref}
\usepackage{xcolor}
\usepackage{subcaption}
\usepackage{enumerate}
\usepackage{ifthen}
\usepackage{wasysym}

\hypersetup{
  colorlinks = true,
  citecolor = purple,
  linkcolor = blue,
  urlcolor = blue
}

\usepackage[margin=3cm]{geometry}

\usepackage{amsmath}
\usepackage{amsthm}
\usepackage{mathtools}
\usepackage{amssymb}
\usepackage{marvosym} 
\usepackage{tikz}
\usepackage{pgfplots}
\usepackage{cleveref}
\usepackage{algorithm2e}

\pgfplotsset{compat=1.18}

\RestyleAlgo{ruled}

\SetAlgorithmName{Model}{model}{List of Models}

\newtheorem{theorem}{Theorem}
\newtheorem{idea}{Idea}

\newtheorem{conjecture}{Conjecture}
\newtheorem{definition}[theorem]{Definition}
\newtheorem{lemma}[theorem]{Lemma}
\newtheorem{claim}[theorem]{Claim}
\newtheorem{corollary}[theorem]{Corollary}

\makeatletter
\newcommand{\dotminus}{\mathbin{\text{\@dotminus}}}
\newcommand{\@dotminus}{%
  \ooalign{\hidewidth\raise1ex\hbox{.}\hidewidth\cr$\m@th-$\cr}%
}
\makeatother
\newcommand{\tshirt}[4]{
    \begin{scope}[scale=#2]
    \fill[#1] (#3 + 1, #4 + 0) -- (#3 + 4, #4 + 0) -- (#3 + 4, #4 + 2) -- (#3 + 3.5, #4 + 4) --  (#3 + 1.5, #4 + 4) -- (#3 + 1, #4 + 2) -- cycle;
    \draw[black] (#3 + 1, #4 + 0) -- (#3 + 4, #4 + 0) -- (#3 + 4, #4 + 2) -- (#3 + 3.5, #4 + 4) --  (#3 + 1.5, #4 + 4) -- (#3 + 1, #4 + 2) -- cycle;
    \fill[#1] (#3 + 1.5, #4 + 4) -- (#3 + 0, #4 + 3) -- (#3 + 0.5, #4 + 2) -- (#3 + 1, #4 + 2) -- cycle;
    \draw[black] (#3 + 1.5, #4 + 4) -- (#3 + 0, #4 + 3) -- (#3 + 0.5, #4 + 2) -- (#3 + 1, #4 + 2) -- cycle;
    
    \fill[#1] (#3 + 3.5, #4 + 4) -- (#3 + 5, #4 + 3) -- (#3 + 4.5, #4 + 2) -- (#3 + 4, #4 + 2) -- cycle;
    \draw[black] (#3 + 3.5, #4 + 4) -- (#3 + 5, #4 + 3) -- (#3 + 4.5, #4 + 2) -- (#3 + 4, #4 + 2) -- cycle;
    \fill[white] (#3 + 2.5, #4 + 4) ellipse (0.5cm and 0.25cm);
   \draw[black] (#3 + 2.5, #4 + 4) +(-0.5,0) arc[start angle=180, end angle=360, x radius=0.5cm, y radius=0.25cm];

    \end{scope}
}

\newcommand{\E}{\mathbb{E}}

\title{Assortment Optimization for Conference Goodies\\ with Indifferent Attendees} 

\author{Fernanda Guti\'errez$^\dagger$ and Bernardo Subercaseaux$^\star$}

\address{$^\dagger$University of Chile, $^\star$Carnegie Mellon University. \\Email: {\rm  \href{mailto:f.gutierrez.b@ing.uchile.cl}{f.gutierrez.b@ing.uchile.cl} and \href{mailto:bersub@cmu.edu}{bersub@cmu.edu}}}

\begin{document}

\begin{abstract}
Conferences such as \textsf{FUN with Algorithms} routinely buy goodies (e.g., t-shirts, coffee mugs, etc) for their attendees. Often, said goodies come in different types, varying by color or design, and organizers need to decide how many goodies of each type to buy. We study the problem of buying optimal amounts of each type under a simple model of preferences by the attendees: they are indifferent to the types but want to be able to choose between more than one type of goodies at the time of their arrival.
The indifference of attendees suggests that the optimal policy is to buy roughly equal amounts for every goodie type. Despite how intuitive this conjecture sounds, we show that this simple model of assortment optimization is quite rich, and even though we make progress towards proving the conjecture (e.g., we succeed when the number of goodie types is $2$ or $3$), the general case with $K$ types remains open. We also present asymptotic results and computer simulations, and finally, to motivate further progress, we offer a reward of $\$100$\textsf{usd} for a full proof.
\end{abstract}

\maketitle

\section{Introduction}\label{sec:intro}
Organizing a good conference can be a challenging task involving a wide variety of sub-problems. Common subroutines of conference organization include scheduling problems concerning the design of a good program, matching problems for assigning reviewers to papers, and even routing problems to provide indications for the attendees to arrive at the conference venue. In this paper we focus on an often overlooked aspect of conference organization, that of buying the right goodies to congratulate attendees.

The conference organizers must buy $N$ goodies (e.g., t-shirts) as they anticipate $N$ attendees.  The goodies come in $K$ different types (e.g., colors). \textbf{How many of each type should the organizers buy?} This is the central question of \emph{assortment optimization}, the family of problems concerning an agent that must select which products to buy in order to maximize some function (e.g., profit) under certain assumptions about their clients.

\paragraph*{\textbf{General model of assortment optimization.}} Let $N, K$ be positive integers. An \emph{initial assortment} is a vector $(n_1, \ldots, n_K) \in \mathbb{N}^K$, with $\sum_{i=1}^K n_i = N$, that represents how many copies of each type of item is bought initially. Let $\mathcal{P}_K$ be set probability distributions over the $K$ item types, meaning that 
\[
\mathcal{P}_K = \left\{ \vec{v} \in [0, 1]^K \; \mid \; \sum_{k=1}^K \vec{v}_k  = 1 \right\}.
\]
Then, clients $a_i, \ldots, a_n$ correspond to functions from $\mathbb{N}^K \to \mathcal{P}_K$. That is, upon seeing how many items of each type remain, a client will choose which type to get, possibly using randomness.

\Cref{fig:ex1} illustrates a concrete example. We now state the model of behavior of the attendees. 

\textbf{Model of indifferent attendees.\footnote{The general terminology of \emph{items} and \emph{clients} corresponds to \emph{goodies} and \emph{attendees} in our particular context.}} There are $N$ attendees $a_1, \ldots, a_N$, and $K$ types of goodies, whose initial stocks are $n_1, \ldots, n_K$. Sequentially, each attendee $a_t$, starting from $a_1$ and finishing by $a_N$, will choose one type $1 \leq k \leq K$ of goodie whose stock at the time $n_k^{(t)}$ is nonzero, get a goodie of that type, and decrease the stock of the type by $1$ ($n_k^{(t+1)} \gets n_k^{(t)}-1$).  In this model, attendees are \emph{indifferent} to the type of goodies they get, meaning that attendee $a_t$ will choose a type of goodie uniformly at random from those whose remaining stock is positive at time $t$. We will say attendee $a_t$ is \emph{unhappy} if at time $t$ there are only goodies of a single type remaining.  
Continuing with the example of~\Cref{fig:ex1}, suppose attendees happen to make the following choices:
\begin{align*}
    a_1 & \mapsto \textsf{red } \quad \quad 
    a_2  \mapsto \textsf{green} \quad \quad
    a_3  \mapsto \textsf{red}, \\
    a_4 & \mapsto \textsf{blue} \quad \quad
    a_5  \mapsto \textsf{green} \quad \quad
    a_6  \mapsto \textsf{blue}, \\
          a_7 & \mapsto \textsf{blue} \quad \quad
           a_8  \mapsto \textsf{green} \quad \quad
            a_9  \mapsto \textsf{green}. 
\end{align*}
In this scenario, there will be $2$ unhappy attendees; $a_8$ and $a_9$ can only choose green, and therefore do not get the pleasure of using randomness. Formally, the model of attendees described corresponds to model $\mathcal{M}$ (\ref{alg:model-M}).
\newcommand{\bettergreen}{green!70!black}
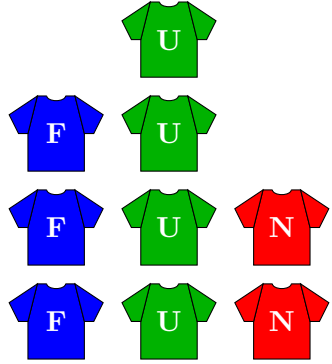
\begin{figure}
    \centering
    \begin{tikzpicture}

        \tshirt{blue}{0.25}{0}{0}
        \tshirt{blue}{0.25}{0}{5}
        \tshirt{blue}{0.25}{0}{10}
        
        \tshirt{\bettergreen}{0.25}{6}{0}
        \tshirt{\bettergreen}{0.25}{6}{5}
        \tshirt{\bettergreen}{0.25}{6}{10}
        \tshirt{\bettergreen}{0.25}{6}{15}

        \tshirt{red}{0.25}{12}{0}
        \tshirt{red}{0.25}{12}{5}

        \node[color=white] (f1) at (0.625, 0.50)
        {\textbf{F}}; 

        \node[color=white] (f2) at (0.625, 1.75)
        {\textbf{F}};  

         \node[color=white] (f3) at (0.625, 3)
        {\textbf{F}};

        \node[color=white] (u1) at (2.125, 0.50)
        {\textbf{U}}; 
        \node[color=white] (u2) at (2.125, 1.75)
        {\textbf{U}}; 
        \node[color=white] (u3) at (2.125, 3)
        {\textbf{U}}; 
        \node[color=white] (u4) at (2.125, 4.25)
        {\textbf{U}}; 

         \node[color=white] (n1) at (3.625, 0.50)
        {\textbf{N}}; 
         \node[color=white] (n2) at (3.625, 1.75)
        {\textbf{N}};

    \end{tikzpicture}
    \caption{An initial assortment of $N = 9$ t-shirts of $K=3$ different types.}
    \label{fig:ex1}
\end{figure}
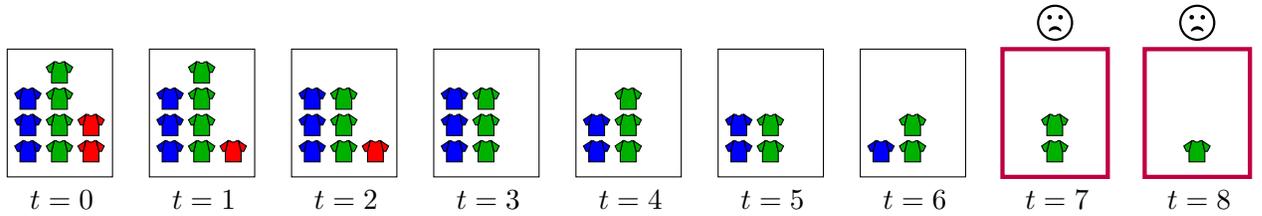
\begin{figure}
    \centering
    \begin{tikzpicture}
        \tshirt{blue}{0.07}{0}{0}
        \tshirt{blue}{0.07}{0}{5}
        \tshirt{blue}{0.07}{0}{10}
        \tshirt{\bettergreen}{0.07}{6}{0}
        \tshirt{\bettergreen}{0.07}{6}{5}
        \tshirt{\bettergreen}{0.07}{6}{10}
        \tshirt{\bettergreen}{0.07}{6}{15}
        \tshirt{red}{0.07}{12}{0}
        \tshirt{red}{0.07}{12}{5}
          \newcommand{\figGap}{27}
        \tshirt{blue}{0.07}{0 + 1*\figGap}{0}
        \tshirt{blue}{0.07}{0+ 1*\figGap}{5}
        \tshirt{blue}{0.07}{0+ 1*\figGap}{10}
        \tshirt{\bettergreen}{0.07}{6+ 1*\figGap}{0}
        \tshirt{\bettergreen}{0.07}{6+ 1*\figGap}{5}
        \tshirt{\bettergreen}{0.07}{6+ 1*\figGap}{10}
        \tshirt{\bettergreen}{0.07}{6+ 1*\figGap}{15}
        \tshirt{red}{0.07}{12+ 1*\figGap}{0}

          \tshirt{blue}{0.07}{0 + 2*\figGap}{0}
        \tshirt{blue}{0.07}{0+ 2*\figGap}{5}
        \tshirt{blue}{0.07}{0+ 2*\figGap}{10}
        \tshirt{\bettergreen}{0.07}{6+ 2*\figGap}{0}
        \tshirt{\bettergreen}{0.07}{6+ 2*\figGap}{5}
        \tshirt{\bettergreen}{0.07}{6+ 2*\figGap}{10}
        \tshirt{red}{0.07}{12+ 2*\figGap}{0}

          \tshirt{blue}{0.07}{0 + 3*\figGap}{0}
        \tshirt{blue}{0.07}{0+ 3*\figGap}{5}
        \tshirt{blue}{0.07}{0+ 3*\figGap}{10}
        \tshirt{\bettergreen}{0.07}{6+ 3*\figGap}{0}
        \tshirt{\bettergreen}{0.07}{6+ 3*\figGap}{5}
        \tshirt{\bettergreen}{0.07}{6+ 3*\figGap}{10}

          \tshirt{blue}{0.07}{0 + 4*\figGap}{0}
        \tshirt{blue}{0.07}{0+ 4*\figGap}{5}
        \tshirt{\bettergreen}{0.07}{6+ 4*\figGap}{0}
        \tshirt{\bettergreen}{0.07}{6+ 4*\figGap}{5}
        \tshirt{\bettergreen}{0.07}{6+ 4*\figGap}{10}

          \tshirt{blue}{0.07}{0 + 5*\figGap}{0}
        \tshirt{blue}{0.07}{0+ 5*\figGap}{5}
        \tshirt{\bettergreen}{0.07}{6+ 5*\figGap}{0}
        \tshirt{\bettergreen}{0.07}{6+ 5*\figGap}{5}

          \tshirt{blue}{0.07}{0 + 6*\figGap}{0}
        \tshirt{\bettergreen}{0.07}{6+ 6*\figGap}{0}
        \tshirt{\bettergreen}{0.07}{6+ 6*\figGap}{5}

        \tshirt{\bettergreen}{0.07}{6+ 7*\figGap}{0}
        \tshirt{\bettergreen}{0.07}{6+ 7*\figGap}{5}

        \tshirt{\bettergreen}{0.07}{6+ 8*\figGap}{0}

    \foreach \i in {0,...,8} { 
       \node (t\i) at (0.625 + 27*0.07*\i, -0.5) {$t = \i$};
       
        \ifthenelse{\i < 7}{
            \draw (-0.1 + 27*0.07*\i , -0.2) -- (1.3 + 27*0.07*\i , -0.2) -- (1.3 + 27*0.07*\i , 1.5) -- (-0.1 + 27*0.07*\i , 1.5) -- cycle;
        }{
        \draw[ultra thick, purple] (-0.1 + 27*0.07*\i , -0.2) -- (1.3 + 27*0.07*\i , -0.2) -- (1.3 + 27*0.07*\i , 1.5) -- (-0.1 + 27*0.07*\i , 1.5) -- cycle;
        \node (f\i) at (0.6 + 27*0.07*\i, 1.85) {\huge \frownie{}};
        } 
    }
    \end{tikzpicture}
    \caption{Illustration of a sequence of attendants over the initial assortment depicted in~\Cref{fig:ex1}. The last two time steps include unhappy attendees. }
\end{figure}
\begin{algorithm}[hbt!]
\caption{$\mathcal{M}$}\label{alg:model-M}
\KwData{Values $n_1, \ldots, n_K$}
\KwResult{Realization of the r.v. $u(n_1, \ldots, n_K)$ (i.e., number of unhappy attendees) }
\For{$i \in \{1, \ldots, K\}$}{
   $ n_i^{(0)} \gets n_i $\
}
$t \gets 0$\\
$\textsf{support}_t \gets \{i \mid n_i^{(t)} \neq 0\}$\\
\While{$|\textsf{support}_t| > 1$}{
  $t \gets t+1$\\
  $a_t \gets \textsf{Uniform}(\textsf{support}_t)$\\
  \For{$i \in \{1, \ldots, K\}$}{
    \eIf{$i = a_t$}{
    $n_i^{(t)} \gets n_i^{(t-1)} - 1$
   } {
    $n_i^{(t)} \gets n_i^{(t-1)}$
   }
  }
     $\textsf{support}_t \gets \{i \mid n_i^{(t)} \neq 0\}$\\
}
\Return{$N - t$, or equivalently, $n_{\textsf{support}_t}^{(t)}$.}
\end{algorithm}
Having formalized the model, we can state the question at the heart of this paper:
\begin{center}
   \emph{Given $N$ and $K$, what initial assortment $(n_1, \ldots , n_K)$ minimizes the expected number of unhappy attendees $\mathbb{E}[u(n_1, \ldots , n_K)]$?}
\end{center}
The worst possible choice for the conference organizers is of course to buy a single goodie type, setting $n_i = N$ and $n_j = 0$ for every $j \neq i$. In contrast, the best choice to minimize unhappy attendees seems intuitively to be for the organizers to get roughly the same amount $n_i$ of each type of goodie. Therefore, we have the following conjecture:

\begin{conjecture}
	For any integers $N, K$, there is an initial assortment $n^\star_1, \ldots, n^\star_K$  that minimizes
    $\mathbb{E}[u(n^\star_1, \ldots, n^\star_K)]$
    such that $\max_{i, j} |n^\star_i - n^\star_j| \leq 1$. 
	\label{conjecture:main}
\end{conjecture}
 
Note that~\Cref{conjecture:main} implies that in case $K$ divides $N$, there will be optimal values $n^\star_i = N/K$. In other words, the conjecture states that $u$ can be minimized in expectation by choosing $n^\star_1, \ldots, n^\star_K$ as close to $N/K$ as possible.
An intuitive reason to believe in~\Cref{conjecture:main} is to consider deterministic approximations of Model~\ref{alg:model-M}. For example, if instead of choosing uniformly at random between the $k \leq K$ remaining types at time $t$, the attendee $a_t$ were to get a $\frac{1}{k}$ fraction of goodie of each remaining type, then the conjecture is almost trivially true. For another example, if attendees were to alternate which of the remaining goodie types to get in a round-robin fashion (i.e., $a_1$ takes from type $1$, $a_2$ from type $2$, $a_K$ from type $K$, $a_{K+1}$ from type $1$, and so on), then the conjecture would also follow easily. Both of these models behave similarly to Model~\ref{alg:model-M} in the sense that in all three of them, the expected number of time steps for any type to decrease its remaining count by $1$ is exactly the number of remaining types.
Another reason to believe in~\Cref{conjecture:main} is that we have verified it computationally for every $N$ up to $30$. The main goal of our paper is therefore to explore different aspects of this conjecture.
 

\subsection{Context and Related Work}

\Cref{conjecture:main} was presented by professor Will Ma to the second author in 2022, during an \emph{Open Problems} session at the Simons Institute for the Theory of Computing. The original presentation of the problem was roughly the following:

\emph{``You have to distribute $N$ rocks into $K$ non-empty piles. Then, round by round, a rock will be removed from a non-empty pile chosen uniformly at random. The process ends when a single non-empty pile remains. The goal is to maximize the expected number of rocks removed, or equivalently, to minimize the expected number of remaining rocks.''}\\

Even though the considered model (Model~\ref{alg:model-M}) is not very realistic when it comes to human behavior, it represents one of the simplest preference models whose behavior we do not fully understand. For example, a slightly more realistic model would be to assume that each attendee $a_t$ has a predefined preference $p_t \in \{1, \ldots, K\}$, and they are fully happy if at time $t$ there are items of type $p_t$ remaining, partially unhappy if there are not but they can still choose among other items, and fully unhappy if they are forced to take an item of a type that does not match their preference. Several similar models can be analyzed, and their behavior is often more complicated than the one studied in this article. For more realistic preference models we refer the reado to e.g.,~\cite{aouadGreedyLikeAlgorithmsDynamic2018,liangAssortmentInventoryPlanning2020, zhangLeveragingDegreeDynamic2023}
In terms of related work, the starting point is \emph{Banach's\footnote{According to William Feller~\cite{Feller1971-uw}, Stefan Banach did not come up with the problem, but rather the problem was
presented by Hugo Steinhaus as a humorous reference to Banach’s smoking habits.} Matchbox Problem}, which can be described as follows: \emph{``A mathematician keeps two matchboxes at all times, one in their left pocket and the other in the right pocket. Any time they want to light up a cigarette, they choose a pocket uniformly at random and pick a match from it. If each matchbox started with $N/2$ matches, what is the probability that when the mathematician opens for the first time an empty matchbox there are exactly $m$ matches in the non-empty box?''} The answer to this textbook exercise (e.g., \cite{Petkovic2009-ix}) is exactly 
\[
\binom{N-m}{m}\left(\frac{1}{2}\right)^{N-m},
\]
and it is a building block of the algebraic approach we show in~\Cref{sec:algebraic}. Another related model is that of Knuth's \emph{toilet paper} problem~\cite{duffyKNUTHGENERALISATIONBANACH2004, knuthToiletPaperProblem1984}, where two toilet paper rolls are available in the bathroom, and there are two kinds of people: \emph{big-choosers}, who choose the roll that has more paper remaining and \emph{little-choosers}, who go for the roll that is closer to running out of paper.
Cacoullos~\cite{cacoullosAsymptoticDistributionGeneralized1967} analyzed the asymptotic distribution of Banach's matchbox problem with $K$ boxes, thus getting closer to our specific setting. Several variants and generalizations of the hypotheses of Banach's matchbox problem have been studied, with applications to different areas of computer science and statistics~\cite{upfal, baringhausMatchboxBottleStorage2002,  kuba2011death, stadjeAsymptoticProbabilitiesSequential1998}.  In assortment optimization, many authors discuss more refined preference models and study them asymptotically or in terms of competitive gaps~\cite{aouadGreedyLikeAlgorithmsDynamic2018,liangAssortmentInventoryPlanning2020, zhangLeveragingDegreeDynamic2023}, but to the best of our knowledge, the conjecture we focus on has not been explicitly discussed before.

\subsection{Organization}
We start~\Cref{sec:k-2} by proving that~\Cref{conjecture:main} holds for $K=2$, which can be done by a very simple induction argument. Then, we discuss why this same argument does not immediately generalize for larger values of $K$. \Cref{sec:k-3} shows how a significantly more sophisticated induction scheme is enough to prove the conjecture for $K=3$. Next,~\Cref{sec:asymptotic} shows asymptotic results regarding the time at which the first type of goodie runs out.~\Cref{sec:experimental} presents simulations and experimental results. Then,~\Cref{sec:algebraic} presents an algebraic approach towards the conjecture, which albeit unsuccessful might be helpful for future work. We conclude in~\Cref{sec:conclusion} by discussing directions towards a full proof.

\section{\texorpdfstring{The simple case, $K=2$.}{The simple case, K=2.}}\label{sec:k-2}

A natural idea when attacking this problem is to consider first the case of two colors of goodies $(K=2)$. This case happens to admit a very simple induction proof.

\subsection{A simple proof by induction}

Consider that instead of minimizing the number of unhappy visitors, we maximize the number of happy visitors, which is clearly equivalent.
As $K=2$, we only have $n_1, n_2 \geq 0$. Let $h(n_1, n_2)$ be the expected number of happy visitors starting with the amounts $n_1, n_2$. Now consider the following intuitive idea.

\begin{idea}
	As long as the values $n_1, n_2, \ldots, n_K$ are not equal, we can improve (or at least not decrease the objective value) by moving $1$ unit from color $\arg\max \{n_1, \ldots, n_K\}$ to $\arg\min \{n_1, \ldots, n_K\}$.
	\label{idea:max-to-min}
\end{idea}

We will see that~\Cref{idea:max-to-min} works in the case $K=2$.
First it is easy to see that we have the following recursive equations for $h$:

\begin{equation}
	h(n_1, 0) = h(0, n_2) = 0,
    \label{eq:base-case}
\end{equation}
\begin{equation}
    \label{eq:recursive}
	h(n_1, n_2) = 1 + \frac{h(n_1 -1, n_2)}{2} +  \frac{h(n_1, n_2 - 1)}{2}, \quad \text{if } n_1, n_2 \geq 0.
\end{equation}

These recursive equations naturally allow for the following proof by induction.

\begin{lemma}
	Let $K = 2$, and let $s \coloneqq \min(n_1, n_2), l \coloneqq \max(n_1, n_2)$\footnote{We will use the letter $s$ to denote a \emph{smaller} size, the letter $m$ to denote a \emph{medium} size (for when we analyze $K=3$), and the letter $l$ to denote a \emph{larger} size.}. If $l > s > 0$, then
	\[
		h(s+1, l) \geq h(s, l+1).
	\]
	\label{lemma:add-to-min-K2}
\end{lemma}
\begin{proof}
	By induction on $N = s + l$. The base case is $N=3$ (with $l = 2 > s = 1$), where $h(2, 2) \geq h(1, 3)$ can be easily checked using~\Cref{eq:base-case,eq:recursive}. Then, for the inductive case, consider
	\begin{align*}
		h(s+1, l) &= 1 + \frac{h(s, l)}{2} + \frac{h(s+1, l-1)}{2}\\
		&\geq 1 + \frac{h(s-1, l+1)}{2} + \frac{h(s+1, l-1)}{2},
	\end{align*}
	where the inequality is trivial in case $l-1=0$, and follows from the inductive hypothesis otherwise. Note as well that
		 \[ h(s+1, l-1) \geq h(s, l), \]
		as if $s = l-1$ then we have equality, and otherwise we can apply the inductive hypothesis. Putting the last two equations together we conclude
		\[
			h(s+1, l) \geq 1 + \frac{h(s-1, l+1)}{2} + \frac{h(s, l)}{2} = h(s, l+1). 
		\]
\end{proof}

Before we can finish the proof for $K=2$, we introduce one last piece of notation.
\begin{definition}
For any initial assortment $S = (n_1, \ldots, n_K)$, we define its \emph{spread} $\sigma(S)$ as $\max(S) - \min(S)$.
\end{definition}
Intuitively,~\Cref{conjecture:main} corresponds to the existence of optimal solutions of minimum spread, and~\Cref{idea:max-to-min} corresponds to progressively decreasing the spread of a current solution. With this notation we can easily finish the proof for $K=2$.

\begin{theorem}
	~\Cref{conjecture:main} holds for $K=2$.
\end{theorem}
\begin{proof}
	Consider an optimal solution $S^\star \coloneqq (n^\star_1, n^\star_2)$ whose spread $\sigma(S^\star)$ is a minimum among the set of optimal solutions, and let $l \coloneqq \max(n^\star_1, n^\star_2), s \coloneqq \min(n^\star_1, n^\star_2)$. If $s = 0$, the conjecture holds trivially by~\Cref{eq:base-case}, so we assume $s > 0$.
    Now, suppose that $l > s + 1$, expecting a contradiction. By applying~\Cref{lemma:add-to-min-K2} to  the pair $(s, l-1)$ we obtain
\[
	h(s+1, l-1) \geq h(s, l),
\]
implying that $(s+1, l-1)$ must also be an optimal solution, but \[\sigma(s+1, l-1) = l-s-2 < l-s = \sigma(s, l),\] thus contradicting the minimality of the spread of $S^\star$.
\end{proof}

\subsection{Why does this not scale to any $K$?}

We now explore what happens if one naively tries to extend this idea for larger values of $K$. First, let us pose the natural generalization of~\Cref{eq:base-case,eq:recursive}. 
\newcommand{\supp}{\operatorname{support}}
We will use notation 
\(
\supp(\vec{n}) = \{ i \mid \vec{n}_i \neq 0\}.
\), as well as $\vec{e}_i$ to denote the vector whose $i$-th component is $1$ and all the rest are $0$.\footnote{To avoid charging the notation, we will leave the dimension of $\vec{e}_i$ implicit as it will be always clear by context (i.e., equal to $K$).}

\newcommand{\vn}{\vec{n}}
\newcommand{\ve}{\vec{e}}

Then, for $\vec{n} = (n_1, \ldots, n_K)$, we have $h(\vec{n}) = 0$ if \, $ |\supp(\vn)| = 1$, and otherwise
\begin{equation}
\label{eq:general-recursive}
    h(\vec{n}) =
 1 + \frac{1}{|\supp(\vec{n})|} \cdot \left(\sum_{i \in \supp(\vec{n})} h(\vec{n} - \vec{e}_i) \right) .
\end{equation}

If we tried to recreate~\Cref{lemma:add-to-min-K2}, letting $s \in \arg\min_i \vn_i$, $l \in \arg\max_i \vn_i$, and assuming $\vn_l > \vn_s > 0$, we would like to prove that
\(
h(\vn + \ve_s) \geq h(\vn + \ve_l).
\) Using~\Cref{eq:general-recursive}, this would be equivalent to showing that
\[
\sum_{i \in \supp(\vn)} h(\vn + \ve_s - \ve_i) \geq \sum_{i \in \supp(\vn)} h(\vn + \ve_l - \ve_i).
\]
It is now tempting to use induction over each $\vn - \ve_i$ to show that 
\[
    h\left((\vn - \ve_i) + \ve_s \right) \geq  h\left((\vn - \ve_i) + \ve_l \right), \quad \forall i \in \supp(\vn).
\]
However, this would require $s \in \arg\min_j (\vn -\ve_i)_j$, and similarly $l\in \arg\max_j (\vn -\ve_i)_j$, which is not true in general. In particular, if there had been two different indices $l \neq l'$ such that $\{l, l'\} \subseteq \arg\max_j \vn_j$, then $l \not\in \arg\max (\vn -\ve_i)_j$ when $i = l'$. 

This issue suggests that the \emph{multiplicities} of values  $\vn_i$ need to be taken into account, and also that we need more refined induction hypotheses so that they are preserved by rounds of the process. We do this carefully for $K=3$ in~\Cref{sec:k-3}. We note as well that a crucial aspect of~\Cref{idea:max-to-min} is that a unit is moved from the maximum stock to the minimum, as opposed to simply moving a unit from any stock that is larger than another. In fact, that idea is simply incorrect as
\(
h(1, 3, 5) > h(2, 2, 5),
\)
or yet, $h(2, 7, 8) > h(3, 6, 8)$, and it is easy to find plenty of such counterexamples computationally. On the other hand, we have not been able to find any counterexamples for~\Cref{idea:max-to-min} despite extensive experimentation.

\section{\texorpdfstring{$K=3$: Wald and the inductive lemmas.}{K=3: Wald and the inductive lemmas.}}\label{sec:k-3}
In order to extend the result of ~\Cref{sec:k-2} to $K=3$ we will require a more sophisticated induction scheme. In particular, we will need to push several inductive lemmas at once. Before introducing said lemmas, we introduce one additional tool that we think can key for generalizing to arbitrary $K$: Wald's equation~\cite{waldGeneralizationsTheoryCumulative1945}.

\begin{definition}
    We define the \emph{``time of the first emptying event''}, $\tau(n_1, \ldots, n_K)$, as the random variable corresponding to the time step in which a color runs out of goodies for the first time. Formally, \[\tau = \min \{ t \mid \exists i \text { such that } n^{(t)}_i = 0 \land n^{(0)}_i > 0\}.\] 
\end{definition}

\begin{lemma}[Consequence of Wald's equation.]
	For any $K > 1$ and positive integers $n_1, \ldots, n_K$, 
	We have 
	\[
		\mathbb{E}[\tau(n_1, \ldots, n_K)] \leq K \cdot \min_c n_c.
	\]
	\label{lemma:wald}
\end{lemma}
Because of the conditions for Wald's equation being mildly technical, we defer the proof of~\Cref{lemma:wald} to the appendix. An important particular case of this lemma that we will use explicitly is that of $K=2$, in which $h(n_1, n_2) = \mathbb{E}[\tau(n_1, n_2)$], and so  we get that $h(n_1, n_2) \leq 2 \min(n_1, n_2)$.

We are now ready to introduce the inductive lemmas:
\begin{enumerate}
	\item \emph{Don't Increase the Lonely Tower} (DILT). Intuitively, if there is a color that starts with more goodies than any other color, then that is the worst color to buy an additional goodie from. In other words, if the goodies from one color are building a \emph{``lonely tower''}, then don't increase it! 
	\begin{lemma}[DILT]
	Let $l > m \geq s$ be non-negative integers. Then
		\[
			h(l, m+1, s) \geq h(l+1, m, s) \quad \text{and} \quad  h(l, m, s+1) \geq h(l+1, m, s).
		\]
		\label{lemma:DILT}
	\end{lemma}
	\item \emph{Don't Increase the Twin Towers} (DITT).
		Intuitively, if two different colors have both the same amount of goodies and more than the third color, then don't increase the size of the \emph{twin towers} and buy from the third color instead.
		\begin{lemma}[DITT]
			Let $l > s$ be non-negative integers. Then
			\[
				h(l, l, s+1) \geq h(l+1, l, s).
			\]
			\label{lemma:DITT}
		\end{lemma}
		
	\item \emph{Dual Transfer to the Poor} (DTP).
		Intuitively, if the color with the fewest goodies has way fewer than the other two, then it is convenient to exchange 1 goodie from each of the top 2 colors for 2 goodies of the bottom color.
	
	 \begin{lemma}[DTP]
	Let $l \geq m \geq s \geq 0$ be integers, such that $m \geq s+2$. Then
	\[
		h(l, m, s) \leq h(l - 1, m - 1, s+2).
	\]
		\label{lemma:DTP}
	\end{lemma}

	\item \emph{Create Third Tower} (CTT). The particular case of the previous lemma when $s = 0$, that requires some special consideration.
	
	\begin{lemma}[CTT]
				Let $l \geq m \geq 2$ be integers. Then
				\[
					h(l, m, 0) \leq h(l-1,m-1, 2).
				\]
				\label{lemma:CTT}
	\end{lemma}
	
	\item \emph{Two Birds in Hand are worth Two in the Bush} (THTB). Intuitively, it's better to directly gain $2$ units of the value rather than $2$ goodies.	

\begin{lemma}[THTB]
		Let $l \geq m > 0$ be integers. Then,
		\[
			h(l, m, 0) \leq h(l-1, m-1, 0) + 2.	
		\]	
		\label{lemma:THTB}
	\end{lemma}

   \item \emph{One Bird in Hand is worth One in the Bush}  (OHOB). Intuitively, it's better to directly gain $1$ units of objective value rather than $1$ extra goodie.	
 \begin{lemma}[OHOB]
		Let $l \geq m > 0$ be integers. Then,
		\[
			h(l, m, 0) \leq h(l-1, m-1, 1) + 1.
		\]
	
		\label{lemma:OHOB}
	\end{lemma}
   
	\end{enumerate}

All these lemmas will be proved by induction over $N$ the sum of the arguments of $h$. However, their proofs have mutual dependencies as we summarize next. Let us use notation, e.g., DITT$(N)$, to refer to the proposition stating that the~\nameref{lemma:DITT} lemma holds for the sum of $h$'s parameters (on the LHS of the lemma) being at most $N$. Then our induction scheme is as follows.

\begin{enumerate}[I)]
	\item \Cref{lemma:wald} $\land$ THTB$(N)$ $\quad \implies$ THTB$(N+1)$.
	\item OHOB$(N)$ $\land$ THTB$(N)$ $\quad \implies$ OHOB$(N+1)$.
	\item CTT$(N)$ $\land$~\Cref{lemma:add-to-min-K2} $\land$ OHOB$(N)$  $\quad \implies$ CTT$(N+1)$.
	\item DTP$(N)$ $\land$ CTT$(N+1)$ $\land$ DITT$(N)$ $\land$ DILT$(N)$ $\quad \implies$ DTP$(N+1)$.
	\item DITT$(N)$ $\land$ DILT$(N)$ $\land$ OHOB$(N)$ $\land$ DTP$(N)$ $\quad \implies$ DITT$(N+1)$. 
	\item DILT$(N)$ $\land$~\Cref{lemma:add-to-min-K2} $\land$ DITT$(N)$ $\quad \implies$ DILT$(N+1)$.
\end{enumerate}

It is easy to check that this scheme is sound, meaning that if we prove all the implication I) through VI), then all lemmas hold for any value of $N$. Let us now show that this is enough to obtain the conjecture for $K=3$.\\

\begin{theorem}
		~\Cref{conjecture:main} holds for $K=3$.
  \label{thm:K-3}
\end{theorem}
\begin{proof}
Let $S^\star = n^\star_1, n^\star_2, n^\star_3$ be an optimal solution such that $\sigma(S^\star) > 1$, and, and assume expecting a contradiction that $S^\star$ is an optimal solution that minimizes the value of $\sigma(S^\star)$.
	 Without loss of generality let us assume 
	\[
	 	l := n^\star_1 \geq m := n^\star_2 \geq s := n^\star_3.
	\]
	Now we proceed by cases. If $l > m + 1$, then $(l-1, m, s)$ satisfies the conditions for the~\nameref{lemma:DILT} lemma, from where 
	\[
		h(l-1, m, s+1) \geq h(l, m, s),
	\]
	and as $\sigma(l-1, m, s+1) = l-s-2 < \sigma(S^\star) = l-s$, we reach a  contradiction. Thus we can safely assume $l \leq m+1$ from now on. If $m \geq s + 2$, then we can apply the~\nameref{lemma:DTP} lemma to obtain a solution $S^{\dagger}$, with $h(S^{\dagger}) \geq h(S^\star)$ and such that $\sigma(S^\dagger) < \sigma(S^\star)$, which is a contradiction again.
	Thus we can safely assume $m < s+2$, and equivalently $s \geq m-1$. If $l = m$, then $l-s = m - s \leq 1$, contradicting $\sigma(S^\star) > 1$. We can thus assume that $l = m+1$. Once again, if $s = m$ the assumption $\sigma(S^\star)$ is contradicted, from where the last case is $s = m-1$, and thus $S^\star = (s+2, s+1, s)$, from where we can obtain another optimal solution with smaller spread by noting
	\[
		h(s+1, s+1, s+1) \geq h(s+2, s+1, s),
	\]
	which is a direct consequence of the~\nameref{lemma:DITT} lemma.
\end{proof}

We now proceed to prove the lemmas by establishing the implications I) through VI).

	\begin{proof}[Proof of~\nameref{lemma:THTB}]
		This statement can be proved by induction over $l+m$. The base case being $l=m=1$, for which the LHS is $1$ and thus the inequality trivially holds. Consider now the case where $m = 1$. In this case we need to prove that $h(l, 1, 0) \leq 2$, which follows from~\Cref{lemma:wald}.
			For $m > 1$ we do the following simple calculation:
		\begin{align*}
			h(l, m, 0) &= 1 + \frac{h(l-1, m, 0)}{2} + \frac{h(l, m-1, 0)}{2}\\
			&\leq 1 + \frac{h(l-2, m-1, 0)+2}{2} + \frac{h(l-1, m-2, 0)+2}{2}\\
			&= h(l-1, m-1, 0) +2.
		\end{align*}
	\end{proof}

	\begin{proof}[Proof of~\nameref{lemma:OHOB}]
		Let us proceed inductively over $l+m$. The base case is $h(1,1,0) \leq h(0,0,1)+1$, which is true by definition.
		Consider first the case where $m = 1$. Then we are trying to prove
		\[
			h(l, 1) \leq h(l-1, 1) + 1,
		\]
		which holds since $h(l, 1) = 1 + h(l-1, 1)/2$.
		Therefore we now assume $m > 1$, and we conclude by the following sequence of calculations:
		
		\begin{align*}
			h(l, m, 0) &= 1 + \frac{h(l-1, m, 0)}{2} + \frac{h(l, m-1, 0)}{2}\\
			&=1 + \frac{h(l-1, m, 0)}{3} + \frac{h(l, m-1, 0)}{3} + \frac{\frac{h(l-1, m) + h(l, m-1)}{2}}{3}\\
			&\leq_{\text{(I. H)}} 1 + \frac{h(l-2, m-1,1)+1}{3}+\frac{h(l-1, m-2,1)+1}{3}+\frac{\frac{h(l-1, m) + h(l, m-1)}{2}}{3}\\
			&= 1 + \frac{h(l-2, m-1,1)+1}{3}+\frac{h(l-1, m-2,1)+1}{3}+\frac{h(l, m, 0) - 1}{3}\\
			&\leq 1 + \frac{h(l-2, m-1,1)+1}{3}+\frac{h(l-1, m-2,1)+1}{3}+\frac{h(l, m, 0) - 1}{3}\\
			&\leq  1 + \frac{h(l-2, m-1,1)+1}{3}+\frac{h(l-1, m-2,1)+1}{3}+\frac{h(l-1, m-1, 0)+1}{3}\tag{by~\Cref{lemma:THTB}}\\
			&= 1 + h(l-1, m-1, 1).
		\end{align*}
	\end{proof}

\begin{proof}[Proof of~\nameref{lemma:CTT}]
We proceed once again by induction over $l+m$. The base case being $h(2,2,0) \leq h(1,1,2)$, which can be easily checked. Note as well that now that the base case has been proven, we can safely assume $l > 2$.  For the inductive case, consider first that,

\[
	h(l, m,0) = 1 + \frac{h(l-1, m, 0)}{3} + \frac{h(l, m-1, 0)}{3} + \frac{h(l, m, 0) - 1}{3},
\]
and that regardless of whether $l-1 \geq m$ or not, we can apply our inductive hypothesis to the first term, obtaining that
\[
	h(l, m, 0) \leq 1 + \frac{h(l-2, m-1, 2)}{3} + \frac{h(l, m-1, 0)}{3} + \frac{h(l, m, 0) - 1}{3}.
\]	
Now, if $m = 2$, then $h(l, m-1, 0) = h(l, 1, 0) \leq h(l-1, 0, 2) = h(l-1, m-2, 2)$, where the inequality follows from~\Cref{lemma:add-to-min-K2}. If, on the other hand, $m > 2$, we can safely apply our inductive hypothesis to $h(l, m-1, 0)$, thus obtaining

\[
	h(l, m, 0) \leq 1 + \frac{h(l-2, m-1, 2)}{3} + \frac{h(l-1, m-2, 2)}{3} + \frac{h(l, m, 0) - 1}{3},
\]

from where it only remains to prove that
\[
	\frac{h(l, m, 0) - 1}{3} \leq \frac{h(l-1, m-1, 1)}{3},
\]
which is an immediate consequence of~\nameref{lemma:OHOB}.
\end{proof}

\begin{proof}[Proof of~\nameref{lemma:DTP}]
	If $s = 0$, then the~\nameref{lemma:CTT} lemma is enough, so we safely assume $s \geq 1$ from now on.
	By induction again on $l+m+s$. Base case is trivial. For the inductive case, let us first  prove
	\[
	h(l-1, m, s) \leq h(l-2, m-1, s+2).
	\]
	This follows by induction taking $l' = l-1$ as long as $l-1 \geq m$. We thus need to be careful about the case $l=m$. In this case, however, we are trying to prove
	\[
		h(m-1, m, s) = h(m, m-1, s) \leq h(m-1, m-2, s+2), 
	\]
	which follows from the inductive hypothesis as long as $m-1 \geq s+2$. Therefore we only need to be careful about the case where $m = s+2$. However, in said case we would be trying to prove:
	\(
		h(s+2, s+1, s) \leq h(s+1, s, s+2),
	\)
	which trivially holds due to the symmetry of $h$.	
	
	Now let us prove
	\[
		h(l, m-1, s) \leq h(l-1, m-2, s+2).
	\]
	If $m > s+2$, then $m-1 \geq s+2$ and thus we can apply our inductive hypothesis (to $(l-1, m-2, s)$, which necessarily satisfies the hypotheses). If $m=s+2$, then we are trying to prove
	\[
		h(l, s+1, s) \leq h(l-1, s, s+2) = h(l-1, s+2, s),
	\]	
		
	which follows from the~\nameref{lemma:DILT} lemma applied to $(l-1, s+1, s)$ as long as $l-1 > s+1$. If $l-1 = s+1$, then
	\[
		h(l, s+1, s) = h(s+2, s+1, s) = h(l-1, s, s+2) = h(s+1, s, s+2),
	\]
	and thus it remains to check the case $l-1 < s+1$, but this cannot happen as we are assuming $m = s+2$, and $l-1 < s+1$ would imply 
	\(
		s \geq l-1 \geq m = s+2,
	\)
	a contradiction.
	
	Finally,  we need to prove
	
	\[
		h(l, m, s-1) \leq h(l-1, m-1, s+1),
	\]
	which follows directly from the inductive hypothesis considering $s \geq 1$.
\end{proof}

\begin{proof}[Proof of~\nameref{lemma:DITT}]
We proceed once again by induction over $s+2l$. Let us consider first the case where $s = 0$. Here we are trying to prove that 
\[
	h(l, l, 1) \geq h(l+1, l, 0).
\]
Given that $l > s = 0$, the RHS can be rewritten as
\[
	h(l+1, l, 0) = 1 + \frac{h(l, l, 0)}{3} + \frac{h(l+1, l-1, 0)}{3} + \frac{h(l+1, l, 0) - 1}{3},
\]
while the LHS can be rewritten as

\[
	h(l, l, 1) = 1 + \frac{h(l-1, l, 1)}{3} + \frac{h(l, l-1, 1)}{3} + \frac{h(l, l, 0)}{3}.
\]
Now notice that we would be done if we had 
\[
	h(l+1, l-1, 0) \leq h(l-1, l, 1) \quad \text{ and } \quad h(l+1, l, 0) - 1 \leq h(l, l-1, 1).
\]
Fortunately, this is the case, as the~\nameref{lemma:DILT} lemma applied over $(l-1, l, 0)$ gives us the first inequality, whereas the second inequality follows exactly from~\nameref{lemma:OHOB}.
We can thus assume safely that $s > 0$, and so what we want to prove is equivalent to

\begin{align*}
	&h(l-1, l, s+1) + h(l, l-1, s+1) + h(l, l, s)\geq h(l, l, s) + h(l+1, l-1, s) + h(l+1, l, s-1),
\end{align*}

which requires only to prove
\[
	h(l, l-1, s+1) \geq h(l+1, l-1, s) \; \text{and} \; h(l-1, l, s+1) \geq h(l+1, l, s-1),
\]	
where the first inequality follows from~\nameref{lemma:DILT}, and the second from the~\nameref{lemma:DTP} lemma.

\end{proof}

\begin{proof}[Proof of~\nameref{lemma:DILT}]
	Once again the proof is by induction on $l+m+s$.
	If $s = 0$ the first part of the lemma comes directly from~\Cref{lemma:add-to-min-K2}. For the second part, we need to prove that
	\[
		h(l, m, 1) \geq h(l+1, m, 0).
	\]
	If $m = 0$ this is trivial, so we assume $m > 0$ and thus we need to prove that
	\begin{align*}
		&\frac{h(l-1, m, 1)}{3}+\frac{h(l, m-1, 1)}{3}+\frac{h(l, m, 0)}{3}\geq \frac{h(l, m, 0)}{3}+\frac{h(l+1, m-1, 0)}{3}+\frac{h(l+1, m, 0) - 1}{3}.
	\end{align*}
	If $l-1 > m$, then $h(l-1, m, 1) \geq h(l, m, 0)$ by the inductive hypothesis. If $l-1 = m$, then $h(l-1, m, 1) \geq h(l, m, 0)$ follows from~\nameref{lemma:DITT}.
	Using the inductive hypothesis again we have $h(l, m-1, 1) \geq h(l+1, m-1, 0)$.  Finally it remains to show that 
	\[
		h(l, m, 0) \geq h(l+1, m, 0) - 1,
	\]
	which is equivalent to showing that
	\[
		h(l, m, 0) \geq h(l, m, 0)/2 + h(l+1, m-1,0)/2,
	\]
	which must hold based on~\Cref{lemma:add-to-min-K2} and using that $l > m$.
	We can now consider the case where $s > 0$. Let us first show the first part of the lemma, which is equivalent (as $s > 0$) to
	\[
		h(l-1, m+1, s) + h(l, m, s) + h(l, m+1, s-1) \geq h(l, m, s) + h(l+1, m-1, s) + h(l+1, m, s-1).
	\]
	Note that we have $h(l, m+1, s-1) \geq h(l+1, m, s-1)$ by inductive hypothesis over $(l, m, s-1)$, and $
		h(l, m, s) \geq h(l+1, m-1, s)$
	 by the inductive hypothesis applied to $(l, m-1, s)$.
	 It remains to show that 
	\[
		h(l-1, m+1, s) \geq h(l, m, s),
	\]
	which is trivially true if $l = m+1$, and otherwise follows by the inductive hypothesis applied to $(l-1, m, s)$ as $l-1 > m$. 	
	\end{proof}

\section{Asymptotic Bounds and Approximations}\label{sec:asymptotic}
The random process studied (i.e.,~Model~\ref{alg:model-M}) can be seen as a sequence of $K-1$ \emph{first emptying events} that are analogous to $\tau$, as described next. 
Throughout the entire process, there are $K-1$ time steps at which a goodie runs out, which we can denote by $\tau, \tau_2, \tau_3, \ldots, \tau_{K-1}$. Each $\tau_i$ is a random variables depending on the initial assortment $n_1, \ldots, n_K$, and moreover, we have that $\tau_{K-1}(n_1, \ldots, n_K) = h(n_1, \ldots, n_K)$. For an example, imagine an initial assortment $\vec{n}^{(0)} := (300, 300, 300, 300)$, and that at time $\tau = 1000$ the first emptying event occurs, leaving the stocks at $\vec{n}^{(1000)} = (24, 30, 0, 14)$. Then, if the next emptying event is at time $\tau_2 = 1039$, leaving the stocks at $\vec{n}^{(1039)} = (x, y, 0, 0)$, what should we expect $x$ and $y$ to be?
We should expect that $x \approx 24 - \frac{39}{3}$ and $y \approx 30 - \frac{39}{3}$, as the $39$ rounds between $\tau$ and $\tau_2$ should be split roughly equally between goodie types. An even better approximation, given we know the last goodie type emptied at $\tau_2$, would be to consider $x \approx 24 - \frac{25}{2}, y \approx 39 - \frac{25}{2}$, as $14$ of the $39$ rounds must have removed from the last goodie type.
More in general, consider the time elapsed between $\tau_i$ and $\tau_{i+1}$ for some $i$, given there were $K-i$ nonzero values $n^{(\tau_i)}_j$ at time $\tau_i$. Let us define random variable $\Delta_i$ as $\tau_{i+1} - \tau_{i}$. Then, we would expect each of the $k$ type of goodies to decrease roughly as:
\begin{equation}
    n^{(\tau_{i+1})}_j \gets n^{(\tau_i)}_j - \frac{\Delta_i}{K-i}. \label{eq:expected-decrease}
\end{equation}
We can use this idea to approximate the general expected duration of the process $h$ based on approximations for the expectations of $\tau$, as we show next.
Consider we have a function $\gamma$ that approximates $\mathbb{E}[\tau]$, meaning that 
\begin{equation}
    \mathbb{E}\left[\tau\left(n_1^{(0)}, \ldots, n_K^{(0)}\right)\right]
    \approx \gamma\left(n_1^{(0)}, \ldots, n_K^{(0)}\right). \label{eq:approx-tau}
\end{equation}

Then, by linearity of expections we have
\begin{equation}
   h\left(n_1^{(0)}, \ldots, n_K^{(0)}\right) = \mathbb{E}\left[\tau\left(n_1^{(0)}, \ldots, n_K^{(0)}\right)\right] + h\left(n_1^{(\tau)}, \ldots, n_K^{(\tau)}\right),
\end{equation}
from where we will define a recursive approximation $\hat{h}(x_1, \ldots, x_K)$ for $h\left(n_1^{(0)}, \ldots, n_K^{(0)}\right)$. 
Using notation $\gamma := \gamma\left(x_1, \ldots, x_K\right)$, $S := \left|\left\{ i \mid x_i \geq 1 \right\}\right|$, and $a \dotminus b := \max\{0, a - b\}$, we pose:
\begin{align}
    \hat{h}\left(x_1, \ldots, x_K\right) &=  \gamma + \hat{h}\left(x_1 \dotminus \frac{\gamma}{S}, \ldots, x_K  \dotminus \frac{\gamma}{S}\right),\\
    \hat{h}\left(n_1^{(0)}, \ldots, n_K^{(0)}\right) &\approx h\left(n_1^{(0)}, \ldots, n_K^{(0)}\right).
\end{align}

As a result, obtaining better approximations for $\tau$ will result in better approximations for $h$. As illustrated in~\Cref{fig:approx-h}, using the approximation $\gamma(x_1, \ldots, x_K) := K\cdot\min\{x_1, \ldots, x_K\}$, we can see that $\hat{h}$ is a good approximation for $h$.
We devote the rest of this section to argue that the approximation $\gamma(n_1, \ldots, n_K) := K\cdot\min_c n_c$ is asymptotically correct, by using Chernoff-Hoeffding bounds.
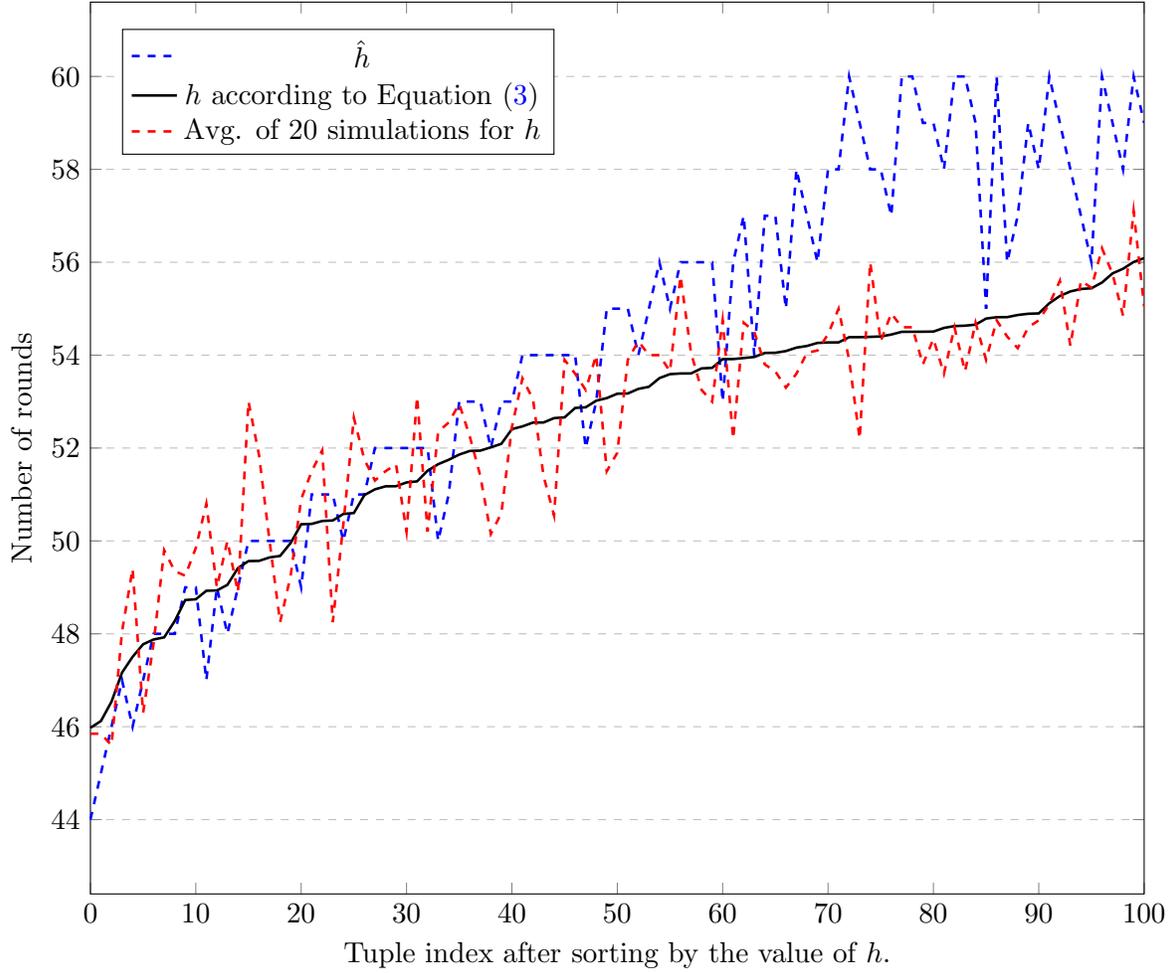
\begin{figure}
    \centering
    \begin{tikzpicture}
        \begin{axis}[
        every axis plot/.append style={line width=0.975pt},
            title={},
            width=\linewidth,
            xlabel={Tuple index after sorting by the value of $h$.},
            ylabel={Number of rounds},
            xmin=0,
            xmax=100,
            legend pos=north west,
            ymajorgrids=true,
            grid style=dashed,
        ]
        \addplot[color=blue, mark='-', dashed] coordinates { (0, 44.0)(1, 45.0)(2, 46.0)(3, 47.0)(4, 46.0)(5, 47.0)(6, 48.0)(7, 48.0)(8, 48.0)(9, 49.0)(10, 49.0)(11, 47.0)(12, 49.0)(13, 48.0)(14, 49.0)(15, 50.0)(16, 50.0)(17, 50.0)(18, 50.0)(19, 50.0)(20, 49.0)(21, 51.0)(22, 51.0)(23, 51.0)(24, 50.0)(25, 51.0)(26, 51.0)(27, 52.0)(28, 52.0)(29, 52.0)(30, 52.0)(31, 52.0)(32, 52.0)(33, 50.0)(34, 51.0)(35, 53.0)(36, 53.0)(37, 53.0)(38, 52.0)(39, 53.0)(40, 53.0)(41, 54.0)(42, 54.0)(43, 54.0)(44, 54.0)(45, 54.0)(46, 54.0)(47, 52.0)(48, 53.0)(49, 55.0)(50, 55.0)(51, 55.0)(52, 54.0)(53, 55.0)(54, 56.0)(55, 55.0)(56, 56.0)(57, 56.0)(58, 56.0)(59, 56.0)(60, 53.0)(61, 56.0)(62, 57.0)(63, 54.0)(64, 57.0)(65, 57.0)(66, 55.0)(67, 58.0)(68, 57.0)(69, 56.0)(70, 58.0)(71, 58.0)(72, 60.0)(73, 59.0)(74, 58.0)(75, 58.0)(76, 57.0)(77, 60.0)(78, 60.0)(79, 59.0)(80, 59.0)(81, 58.0)(82, 60.0)(83, 60.0)(84, 59.0)(85, 55.0)(86, 60.0)(87, 56.0)(88, 57.0)(89, 59.0)(90, 58.0)(91, 60.0)(92, 59.0)(93, 58.0)(94, 57.0)(95, 56.0)(96, 60.0)(97, 59.0)(98, 58.0)(99, 60.0)(100, 59.0) };

        \addplot[color=black, mark='-'] coordinates { (0, 45.976408626269674)(1, 46.12287165082226)(2, 46.53920421678334)(3, 47.165058682792335)(4, 47.50396872343833)(5, 47.77478531879051)(6, 47.87809631413966)(7, 47.92378542293321)(8, 48.27433350432007)(9, 48.72728898434274)(10, 48.74366401183474)(11, 48.92712533372472)(12, 48.93339289638899)(13, 49.054551857361865)(14, 49.4157950733903)(15, 49.56599947785703)(16, 49.5700140341481)(17, 49.644503413314794)(18, 49.67682277301134)(19, 49.95519505997238)(20, 50.35907954012349)(21, 50.36661034752679)(22, 50.42926828389239)(23, 50.44040665167286)(24, 50.58064146200401)(25, 50.600512221920525)(26, 50.98612800525728)(27, 51.11082297879122)(28, 51.1751979056535)(29, 51.17783514296492)(30, 51.25945090128175)(31, 51.28198286549036)(32, 51.512797635188605)(33, 51.651945339797386)(34, 51.74720305266341)(35, 51.859201756159095)(36, 51.93685307708679)(37, 51.944362780670126)(38, 52.015086422129485)(39, 52.091081863752905)(40, 52.40792114144442)(41, 52.46582880237622)(42, 52.54813507525275)(43, 52.54986601438202)(44, 52.64435157675961)(45, 52.65998166663149)(46, 52.86300191091007)(47, 52.87707897342995)(48, 53.019072219604126)(49, 53.07501425693806)(50, 53.16949014594548)(51, 53.174591582044386)(52, 53.273498659512725)(53, 53.31875842426261)(54, 53.50576137760429)(55, 53.59070438856077)(56, 53.60497647242384)(57, 53.606142744438024)(58, 53.71474644936977)(59, 53.72581976484488)(60, 53.913955535515)(61, 53.915066524799286)(62, 53.93812804490791)(63, 53.96046494241949)(64, 54.04709534076405)(65, 54.05073446375285)(66, 54.087770010496854)(67, 54.161992393946505)(68, 54.19746361798793)(69, 54.26317609425418)(70, 54.27332022085987)(71, 54.27417905575799)(72, 54.38649813163828)(73, 54.386989022314246)(74, 54.393590997147534)(75, 54.40207616029504)(76, 54.44342954090535)(77, 54.50222977619359)(78, 54.50299118217261)(79, 54.50433616208786)(80, 54.50728574567164)(81, 54.58590813554514)(82, 54.6264387636261)(83, 54.63408510760178)(84, 54.65674645876372)(85, 54.78754019512651)(86, 54.81663600626329)(87, 54.818300773750664)(88, 54.86405061597931)(89, 54.89109259118552)(90, 54.89823303633792)(91, 55.11735157458611)(92, 55.27347617166674)(93, 55.37146012668457)(94, 55.42381249537505)(95, 55.44007187060478)(96, 55.562809178110605)(97, 55.76004471512821)(98, 55.86010864242161)(99, 56.00531008152717)(100, 56.08927527700829) };
       
        \addplot[color=red, mark='-', dashed] coordinates { (0, 45.85)(1, 45.85)(2, 45.6)(3, 48.05)(4, 49.4)(5, 46.3)(6, 47.8)(7, 49.8)(8, 49.35)(9, 49.25)(10, 49.85)(11, 50.8)(12, 49.0)(13, 50.0)(14, 48.9)(15, 53.0)(16, 51.85)(17, 50.0)(18, 48.25)(19, 49.25)(20, 50.9)(21, 51.5)(22, 51.95)(23, 48.25)(24, 50.35)(25, 52.65)(26, 51.75)(27, 51.3)(28, 51.5)(29, 51.65)(30, 50.2)(31, 53.1)(32, 50.2)(33, 52.35)(34, 52.55)(35, 52.95)(36, 52.25)(37, 51.4)(38, 50.15)(39, 50.6)(40, 52.45)(41, 53.5)(42, 53.05)(43, 51.4)(44, 50.55)(45, 53.9)(46, 53.6)(47, 53.25)(48, 54.0)(49, 51.5)(50, 51.9)(51, 53.9)(52, 54.3)(53, 54.0)(54, 54.0)(55, 53.65)(56, 55.7)(57, 54.05)(58, 53.25)(59, 53.0)(60, 54.75)(61, 52.2)(62, 54.7)(63, 54.5)(64, 53.8)(65, 53.65)(66, 53.3)(67, 53.6)(68, 54.05)(69, 54.1)(70, 54.45)(71, 55.0)(72, 53.9)(73, 52.2)(74, 56.0)(75, 54.3)(76, 54.9)(77, 54.6)(78, 54.6)(79, 53.8)(80, 54.35)(81, 53.6)(82, 54.6)(83, 53.65)(84, 54.7)(85, 53.9)(86, 54.75)(87, 54.4)(88, 54.15)(89, 54.6)(90, 54.75)(91, 55.1)(92, 55.6)(93, 54.2)(94, 55.6)(95, 55.45)(96, 56.3)(97, 55.75)(98, 54.85)(99, 57.1)(100, 55.05) };

\legend{$\hat{h}$,$h$ according to \Cref{eq:general-recursive}, Avg. of 20 simulations for $h$}
        
        \end{axis}
    \end{tikzpicture}
    \caption{Illustration of the relationship between $\hat{h}$, $h$ computed acocrding to~\Cref{eq:general-recursive}, and the average of 20 simulations of according to Model~\ref{alg:model-M}.}
    \label{fig:approx-h}
\end{figure}

The core idea is showing that the type of goodie $m := \arg\min_c n_c$ is the most important one to approximate $\mathbb{E}[\tau]$. 
As a first step, we will need a lemma stating that types that start with fewer items are more likely to empty first. 

\begin{lemma}
Let $(n_1, \ldots, n_K)$ be an initial assortment, and let $i, j$ be indices such that $n_i < n_j$. Let $p$ the random variable defined as the type of goodie (i.e., a number in $\{1, \ldots, K\}$) that is emptied first. Then, for all $t$ we have 
\[
\Pr[p=i~\,\mathrm{and} \,~\tau\leq t] \geq \Pr[p=j~\,\mathrm{and}\,~\tau\leq t].
\]\label{lemma:prob-inequality}
\end{lemma}

Despite it being intuitive, our proof of~\Cref{lemma:prob-inequality} uses some non-trivial calculations and is deferred to~\Cref{app:proofs}.


We now give a lower bound for $\mathbb{E}[\tau]$ in terms of $n_m$ and $K$.

\begin{lemma}
For any $K>1$ and positive integers $n_1,\dots,n_K$, we have
\[
\mathbb{E}[\tau(n_1,\dots,n_K)] \geq (K-2)\cdot n_m - \delta,
\]
where $\delta$ is a real number such that 
\( 0 \geq \delta \geq \left(K - 2\right)\left(\sqrt{3\ln K \cdot (2n_m + 3\ln K)} - 3\ln K \right).
\)
\label{lemma:chernoff-hoeffding}
\end{lemma}
\begin{proof}
    Consider the sequence of independent Bernoulli random variables defined as 
        \[A_n:= \begin{cases} 1 & \text{with prob. } 1/K \\
            0 & \text{ with prob.} (K-1)/K \end{cases}\]
        and let $Y_n = \sum_{i=1}^{n} A_i$. Thus $Y_n$ is a finite sum. By the Chernoff-Hoeffding bound we get that for any $d\geq 0$ and $\mu = \mathrm{E}[Y_n]$, we have
    \begin{equation}
    \Pr[\left|Y_n-\mu \right| \geq d ] \leq 2 \cdot \exp{\left(\frac{-d^2}{3\mu}\right)}.
    \label{eq:chernoff-hoeffding}   
\end{equation}
Let $0 \leq x \leq K\cdot n_m$ be a real number we will determine later. We will take $n := \lfloor K \cdot n_m - x \rfloor$. Given that $\E[A_i] = \frac{1}{K}$ for every $i$, we have $\mu = \E[Y_n] = \frac{n}{K}$.
 Replacing into~\Cref{eq:chernoff-hoeffding} with $d = x/K$, we obtain 
    \begin{equation}
        \Pr\left[\left|Y_n - \frac{n}{K}\right| \geq  \frac{x}{K}\right] \leq 2 \cdot \exp \left(\frac{-x^2/K^2}{ \left(\frac{3 \lfloor K \cdot n_m - x \rfloor}{K} \right)}\right)  \leq  2 \cdot \exp \left(\frac{-x^2}{3K \left(K \cdot n_m - x\right)}\right).
    \end{equation}
    On ther other hand, using that $(n+x)/K  \leq (K \cdot n_m  -x + x)/K = n_m/K$, we have 
    \begin{equation}
        \Pr\left[\left|Y_n - \frac{n}{K}\right|\geq  \frac{x}{K}\right] \geq   \Pr\left[Y_n  \geq \frac{n}{K} + \frac{x}{K}\right]  \geq \Pr[Y_n \geq n_m].
        \label{eq:gt-nm}
    \end{equation}
    We now find a value of $x$ for which the $\textrm{RHS}$ is lower than $\frac{2}{K^2}$. This is equivalent to $x$ verifying the inequality
    \begin{equation} \frac{x^2}{3 K  (K \cdot n_m - x)} \geq 2 \ln K. \label{eq:x-solutions}\end{equation}
   By the quadratic formula, we obtain that the positive solutions are
    \[x \geq K\left(\sqrt{3\ln K \cdot (2n_m + 3\ln K)} - 3\ln K \right).\]
    
    Then, combining the previous calculation with~\Cref{eq:gt-nm}, and recalling that $p$ is the random variable denoting the type that runs off for the first time, for any value of $x$ that satisfies~\Cref{eq:x-solutions}, we have
    \[ 
        \frac{2}{K^2}\geq \Pr[Y_n \geq n_m] \geq \Pr[p=m~\text{and}~ \tau \leq n ].
    \] 
Let us consider from now on a fixed value $x^*$ that satisfies~\Cref{eq:x-solutions}.
    
    Using the abbreviation $P_t^i:=\Pr[p=i~\text{and}~ \tau\leq t ]$. In this notation,~\Cref{lemma:prob-inequality}~establishes that $P_t^m\geq P_t^i$ for every $i \in \{1, \ldots, K\}$ and $t \in \mathbb{N}$. Now, note that
    \begin{equation}
        \Pr[\tau \leq n] = \sum_{i=1}^K P_n^i \leq  \sum_{i=1}^K P_n^m = K \cdot P_n^m \leq K \cdot \frac{2}{K^2} = \frac{2}{K},
    \end{equation}
    from where 
    \begin{equation}
        1 - \frac{2}{K} \leq 1-\Pr[\tau \leq n] = \Pr[\tau>n].\label{eq:almost-final}
    \end{equation}
    We are now ready to conclude by the following chain of inequalities:

    \begin{align*}
    \mathbb{E}[\tau(n_1,\dots,n_K)] & = \sum_{t=1}^{\infty} t \cdot \Pr[\tau = t]\\
    & = \sum_{t\leq n} t \cdot \Pr[\tau=t] + \sum_{t> n} t \cdot \Pr[\tau = t]\\
    & \geq \sum_{t > n} t \cdot \Pr[\tau=t] \\
    & \geq n \cdot \sum_{t > n} \cdot \Pr[\tau=t]\\
    & \geq n \cdot \Pr[\tau >n] \\
    & \geq n \cdot \left(1 - \frac{2}{K}\right)\tag{By~\Cref{eq:almost-final}}\\
    &\geq (K - 2) \cdot n_m - x^* \left(1 - \frac{2}{K}\right).\\
    \end{align*}
    Note that the last inequality holds for any such fixed $x^*$, therefore we get
    $$\E[\tau(n_1,\dots,n_K)]\geq (K - 2) \cdot n_m - x^* \left(1 - \frac{2}{K}\right) = (K-2)\cdot n_m - \delta$$
    for all $\delta$ such that 
    \begin{align*}
    \delta &\geq \left(1-\frac{2}{K}\right) \cdot K \left(\sqrt{3\ln K (2n_m + 3\ln K)} - 3\ln K \right)\\ 
    &= \left(K - 2\right)\left(\sqrt{3\ln K \cdot (2n_m + 3\ln K)} - 3\ln K \right).
    \end{align*}
    This concludes the proof.
\end{proof}

As a direct corollary of the previous lemma we get

\begin{corollary}\label{cor:chernof}
Asymptotically, we have \(
\mathbb{E}[\tau(n_1, \ldots, n_K)] \in \Omega\left(K\cdot n_m \left(1-\sqrt{\frac{\ln K}{n_m}} \right) \right) 
\), and thus as long as $\ln K \ll n_m$, the bounds are asymptotically tight.
\end{corollary}


\section{Simulations and Experimental Results}\label{sec:experimental}



In this section we present experimental results concerning the value of~$\mathbb{E}[\tau(n_1, \ldots, n_K)]$, with the particular goal of studying how tight the bounds derived in~\Cref{sec:k-3} and~\Cref{sec:asymptotic} are.
We directly implemented Model~\ref{alg:model-M} in Python, and ran extensive simulations. 
We use notation $\hat{\tau}$ for the average result of $10\,000$ simulations.  Our main findings are the following:\\

\paragraph*{\textbf{Tightness of upper bound}} 
The behavior of $\mathbb{E}[\tau(n_1, \ldots, n_K)]$ is very close to $K \cdot \min_c n_c$ in a wide variety of settings. Let us denote by $\mathcal{U}^K([a, b])$ the distribution where each $n_i$ in the initial assortment is drawn independently and uniformly at random from $\{a, \ldots, b\}$. 
For different values of $K, a$ and $b$, we generated $30$ random samples according to $\mathcal{U}^K([a,b])$, and for each of those studied the comparison between the theoretical bound $K \cdot n_m$ (where $n_m$ is now the result of sampling a random variable), and $\hat{\tau}$.  The results are depicted in~\Cref{fig:approx-tau}. 
Another interesting distribution consisting in fixing a value of $N = \sum_{i=1}^K n_i$, and then sampling uniformly at random from the set of $K$-tuples with sum $N$. We denote this distribution by $\mathcal{S}(K, N)$, and present the results in~\Cref{fig:S-distribution}.  The figures presented in the later paragraphs also confirm this finding.\\

\paragraph*{\textbf{The role of $K$ in the lower bound}}
As mentioned in~\Cref{cor:chernof}, the quality of our lower bound depends on $K$ being sufficiently large. As depicted in~\Cref{fig:small-K,fig-all-three-tuples,fig:zoomed}, for $K=3$ the lower bound we obtain is roughly $\min_c n_c$, whereas the correct bound is basically $3 \cdot \min_c n_c$, as mentioned in the previous paragraph. In contrast, as~\Cref{fig:K-5} illustrates, already for $K=5$ the lower bound improves over $\min_c n_c$, which is a trivial lower bound. As depicted in~\Cref{fig:large-K}, for $K=50$ and $n_m ~ 200$ the lower bound is roughly $75\%$ of the correct value. 


%
%


%


\section{An Algebraic Approach}\label{sec:algebraic}
We now present an algebraic proof that does not quite work for the case $K=2$. This section is inspired by the analysis of the \emph{Banach's matchbox problem}. 

\begin{theorem}
	Let $N = n_1 + n_2$ be even. Then $\mathbb{E}[u(n_1, n_2)]$ is minimized when $n_1 = n_2 = N/2$.
	\label{thm:K-2}
\end{theorem}

\begin{proof}[An algebraic approach to proving~\Cref{thm:K-2}]
\label{proof:attempt-K-2}

	To prove this, we will find an explicit formula for $\E[u(n_1, n_2)]$, and then show that the pair $(N/2, N/2)$ is a minimizer.
	Let $\downarrow_1$ be the event denoting that all unhappy visitors get goodies of color $1$.
	 Similarly, let $\downarrow_2$ be the event when the unhappy visitors get goodies of color $2$.
	 By the law of total probabilities, we have
	\begin{equation}
		\E[u(n_1, n_2)] = \E[u(n_1, n_2) \mid \, \downarrow_1]\cdot \Pr[\downarrow_1] + \E[u(n_1, n_2) \mid \, \downarrow_2]\cdot \Pr[\downarrow_2]. 
	\end{equation}
	
	To simplify our notation, let
	\[
	E_1 \coloneqq \E[u(n_1, n_2) \mid \, \downarrow_1]\cdot \Pr[\downarrow_1] \; \quad ; \quad 	E_2 \coloneqq \E[u(n_1, n_2) \mid \, \downarrow_2]\cdot \Pr[\downarrow_2].
	\]
	
	Also, let $\downarrow_1^r$ denote the event when the unhappy visitors get goodies of color $1$ and there are exactly $r \geq 1$ unhappy visitors.
	Note that 
	\begin{equation}
		E_1 = \sum_{r = 1}^{n_1} r \cdot \Pr[\downarrow_1^r \mid \; \downarrow_1] \cdot \Pr[\downarrow_1] =   \sum_{r = 1}^{n_1} r \cdot \Pr[\downarrow_1^r],
	\end{equation}
	and analogously,
	\begin{equation}
		E_2 =  \sum_{r = 1}^{n_2} r \cdot \Pr[\downarrow_2^r].
	\end{equation}
	Now let us compute $\Pr[\downarrow_1^r]$ (the same calculations will hold for $\Pr[\downarrow_2^r]$). Indeed, for $\downarrow_1^r$ to occur we need that right before the last happy visitor, there are $r$ goodies remaining of color $1$ and $1$ single goodie remaining of color $2$. After that, we simply need the event where the next visitor takes the last goodie of color $2$, which occurs with probability $1/2$. Therefore,
	\begin{equation}
		\Pr[\downarrow_1^r] = {N-r-1 \choose n_1 - r}\left(\frac{1}{2}\right)^{N-r-1} \cdot \frac{1}{2} = {N-r-1 \choose n_1 - r}\left(\frac{1}{2}\right)^{N-r},
		\label{eq:direct_prob}
	\end{equation}
	as we need that exactly $n_1 - r$ visitors among the initial $N-r-1$ choose goodie of color $1$, and each such sequence of visitor choices has probability $\left(\frac{1}{2}\right)^{N-r-1}$.
	Now note that by rewriting the combinatorial number in terms of factorial we obtain for $r \geq 1$ that
	\begin{equation}
		\Pr[\downarrow_1^{r+1}] = \frac{2 \cdot (n_1 - r)}{N-r-1} \Pr[\downarrow_1^{r}],
	\end{equation}
	and thus 
	\[
		(N-r-1)\Pr[\downarrow_1^{r+1}] = 2 \cdot (n_1 - r)\Pr[\downarrow_1^{r}],
	\]
	which in turn gives us
	\begin{equation}
		N\Pr[\downarrow_1^{r+1}] - (r+1)\Pr[\downarrow_1^{r+1}] = 2n_1\Pr[\downarrow_1^{r}] - 2r\Pr[\downarrow_1^{r}].
		\label{eq:linearized}
	\end{equation}
	By summing over~\Cref{eq:linearized} from $r=1$ to $r=n_1$, and using that $\Pr[\downarrow_1] = \sum_{r=1}^{n_1}\Pr[\downarrow^r_1]$ we obtain
	\[
	N(\Pr[\downarrow_1] - \Pr[\downarrow_1^1])- (E_1 - \Pr[\downarrow_1^1]) = 2n_1\Pr[\downarrow_1] - 2E_1,
	\]
	from where we get
	\begin{equation}
		E_1 = (2n_1-N)\Pr[\downarrow_1] + (N-1)\Pr[\downarrow_1^1],
		\label{eq:E_1}
	\end{equation}
	and analogously,
	\begin{equation}
		E_2 = (2n_2-N)\Pr[\downarrow_2] + (N-1)\Pr[\downarrow_2^1].
		\label{eq:E_2}
	\end{equation}
	
	Now from~\Cref{eq:direct_prob} we have that 
	\[
	\Pr[\downarrow_1^1] = {N-2 \choose n_1 - 1} \left(\frac{1}{2}\right)^{N-1} = {N-2 \choose n_2 - 1} \left(\frac{1}{2}\right)^{N-1} = \Pr[\downarrow_2^1].
	\]

	Recalling that $\E[c(n_1, n_2)] = E_1 + E_2$, 
	and up adding~\Cref{eq:E_1} and~\Cref{eq:E_2}, we get
	\begin{equation}
		\E[c(n_1, n_2)] = (n_1-n_2)(\Pr[\downarrow_1]-\Pr[\downarrow_2]) + (N-1){N-2 \choose n_1 - 1}\left(\frac{1}{2}\right)^{N-2}.
		\label{eq:main-K-2}
	\end{equation}
Let us take a moment to inspect the previous equation, as it is the crux of our approach;
the term
\[
(n_1-n_2)(\Pr[\downarrow_1]-\Pr[\downarrow_2])
\]
is clearly $0$ when $n_1 = n_2$. Moreover, it is not hard to see that it is always non-negative as we prove next.

\begin{claim}
	For all values $n_1, n_2$ we have that $(n_1-n_2)(\Pr[\downarrow_1]-\Pr[\downarrow_2]) \geq 0$.
	\label{claim:sign-probs}
\end{claim} 
\begin{proof}[Proof of~\Cref{claim:sign-probs}]
	Assume without loss of generality that $n_1 \geq n_2 > 0$, and let us see that $\Pr[\downarrow_1] \geq \Pr[\downarrow_2]$. Indeed, the events $\downarrow_1$ and $\downarrow_2$ depend on the values of $n_1$ and $n_2$, so we can write
	\[
		\Pr[\downarrow_1]  = \Pr[\downarrow_1 \mid n_1, n_2] \quad ; \quad \Pr[\downarrow_2]  = \Pr[\downarrow_2 \mid n_1, n_2].
	\]	
	We will now prove the claim by induction on $N = n_1 + n_2$. The base case $N=2$ with $n_1 = n_2$ is trivial.
	For the inductive case, if $n_1 = n_2$, then the claim holds trivially, so we can safely assume $n_1 > n_2$. Now note that
	\[
	\Pr[\downarrow_1 \mid n_1, n_2] = \frac{\Pr[\downarrow_1 \mid n_1-1, n_2]}{2} + \frac{\Pr[\downarrow_1 \mid n_1, n_2-1]}{2},
	\]
	and as $n_1 > n_2$ we have $n_1 - 1 \geq n_2$, which allows us to use the inductive hypothesis on both terms of the previous equation, thus yielding
	\begin{align*}
		\Pr[\downarrow_1 \mid n_1, n_2] &= \frac{\Pr[\downarrow_1 \mid n_1-1, n_2]}{2} + \frac{\Pr[\downarrow_1 \mid n_1, n_2-1]}{2}\\
		&\geq \frac{\Pr[\downarrow_2 \mid n_1-1, n_2]}{2} + \frac{\Pr[\downarrow_2 \mid n_1, n_2-1]}{2}\\
		&= \Pr[\downarrow_2 \mid n_1, n_2],
	\end{align*}
	thus concluding the proof.
\end{proof}

Therefore, if we go back to considering~\Cref{eq:main-K-2}, we see that $n_1 = n_2 = N/2$ minimizes the term
\[
	(n_1-n_2)(\Pr[\downarrow_1]-\Pr[\downarrow_2]).
\]
Unfortunately, $n_1 = n_2 = N/2$ also maximizes the second term of~\Cref{eq:main-K-2}:
\[
	(N-1){N-2 \choose n_1 - 1}\left(\frac{1}{2}\right)^{N-2}.
\]

\begin{figure}
    \centering
        \begin{tikzpicture}

        \begin{axis}[
every axis plot/.append style={line width=0.975pt},
    title={},
    width=\linewidth,
    height=\axisdefaultheight,
    xlabel={Value of $n_1$},
    ylabel={},
    ymax=100,
    xmin=0,
    xmax=100,
    legend pos=north west,
    ymajorgrids=true,
    grid style=dashed,
]
\addplot[
    color=red,
    mark='-',
]
    coordinates {
    (2, 96.0)(3, 94.0)(4, 92.0)(5, 90.0)(6, 88.0)(7, 86.0)(8, 84.0)(9, 82.0)(10, 80.0)(11, 78.0)(12, 75.99999999999997)(13, 73.99999999999976)(14, 71.99999999999835)(15, 69.99999999999001)(16, 67.9999999999441)(17, 65.99999999971028)(18, 63.99999999860439)(19, 61.999999993728714)(20, 59.99999997363584)(21, 57.999999896039924)(22, 55.99999961459018)(23, 53.99999865386971)(24, 51.9999955621496)(25, 49.99998616713763)(26, 47.99995917124035)(27, 45.99988573019758)(28, 43.999696369240226)(29, 41.99923318076021)(30, 39.99815750281425)(31, 37.995784266195805)(32, 35.990807540886124)(33, 33.98088491955586)(34, 31.962072594385035)(35, 29.92816111927482)(36, 27.87006168552384)(37, 25.77552673947113)(38, 23.629623251824572)(39, 21.41644764933532)(40, 19.12249412825969)(41, 16.741789673490967)(42, 14.282374615468974)(43, 11.773024873662669)(44, 9.268486114679433)(45, 6.851201089623775)(46, 4.627813251073839)(47, 2.719719098407121)(48, 1.2484586256812362)(49, 0.3183569495487151)(50, 0.0)(51, 0.3183569495487151)(52, 1.2484586256812362)(53, 2.719719098407121)(54, 4.627813251073839)(55, 6.851201089623775)(56, 9.268486114679433)(57, 11.773024873662669)(58, 14.282374615468974)(59, 16.741789673490967)(60, 19.12249412825969)(61, 21.41644764933532)(62, 23.629623251824572)(63, 25.77552673947113)(64, 27.87006168552384)(65, 29.92816111927482)(66, 31.962072594385035)(67, 33.98088491955586)(68, 35.990807540886124)(69, 37.995784266195805)(70, 39.99815750281425)(71, 41.99923318076021)(72, 43.999696369240226)(73, 45.99988573019758)(74, 47.99995917124035)(75, 49.99998616713763)(76, 51.9999955621496)(77, 53.99999865386971)(78, 55.99999961459018)(79, 57.999999896039924)(80, 59.99999997363584)(81, 61.999999993728714)(82, 63.99999999860439)(83, 65.99999999971028)(84, 67.9999999999441)(85, 69.99999999999001)(86, 71.99999999999835)(87, 73.99999999999976)(88, 75.99999999999997)(89, 78.0)(90, 80.0)(91, 82.0)(92, 84.0)(93, 86.0)(94, 88.0)(95, 90.0)(96, 92.0)(97, 94.0)(98, 96.0)(99, 98.0)
    };

\addplot[
    color=blue,
    mark='-',
]
    coordinates {
(2, 3.0614114009817026e-26)(3, 1.4847845294761258e-24)(4, 4.7513104943236025e-23)(5, 1.1284362424018556e-21)(6, 2.1214601357154885e-20)(7, 3.288263210359007e-19)(8, 4.321717362186124e-18)(9, 4.915953499486716e-17)(10, 4.915953499486716e-16)(11, 4.375198614543177e-15)(12, 3.5001588916345416e-14)(13, 2.537615196435043e-13)(14, 1.678730053026259e-12)(15, 1.019228960765943e-11)(16, 5.7076821802892805e-11)(17, 2.9608601310250645e-10)(18, 1.42817959261209e-09)(19, 6.426808166754405e-09)(20, 2.7060244912650122e-08)(21, 1.0688796740496798e-07)(22, 3.9701245036130966e-07)(23, 1.3895435762645838e-06)(24, 4.591535295482973e-06)(25, 1.434854779838429e-05)(26, 4.2471701483217495e-05)(27, 0.00011924746954903375)(28, 0.0003179932521307566)(29, 0.0008063400321887043)(30, 0.0019463380087313554)(31, 0.004476577420082117)(32, 0.009819589179534967)(33, 0.020559764844651336)(34, 0.04111952968930267)(35, 0.0786108655824904)(36, 0.14374558277941102)(37, 0.2515547698639693)(38, 0.42152420896124587)(39, 0.6766572828062104)(40, 1.0410112043172468)(41, 1.5354915263679392)(42, 2.17215874461806)(43, 2.947929724838796)(44, 3.8391642928133156)(45, 4.798955366016645)(46, 5.758746439219974)(47, 6.635077419101274)(48, 7.340936719005665)(49, 7.799745263943518)(50, 7.958923738717876)(51, 7.799745263943518)(52, 7.340936719005665)(53, 6.635077419101274)(54, 5.758746439219974)(55, 4.798955366016645)(56, 3.8391642928133156)(57, 2.947929724838796)(58, 2.17215874461806)(59, 1.5354915263679392)(60, 1.0410112043172468)(61, 0.6766572828062104)(62, 0.42152420896124587)(63, 0.2515547698639693)(64, 0.14374558277941102)(65, 0.0786108655824904)(66, 0.04111952968930267)(67, 0.020559764844651336)(68, 0.009819589179534967)(69, 0.004476577420082117)(70, 0.0019463380087313554)(71, 0.0008063400321887043)(72, 0.0003179932521307566)(73, 0.00011924746954903375)(74, 4.2471701483217495e-05)(75, 1.434854779838429e-05)(76, 4.591535295482973e-06)(77, 1.3895435762645838e-06)(78, 3.9701245036130966e-07)(79, 1.0688796740496798e-07)(80, 2.7060244912650122e-08)(81, 6.426808166754405e-09)(82, 1.42817959261209e-09)(83, 2.9608601310250645e-10)(84, 5.7076821802892805e-11)(85, 1.019228960765943e-11)(86, 1.678730053026259e-12)(87, 2.537615196435043e-13)(88, 3.5001588916345416e-14)(89, 4.375198614543177e-15)(90, 4.915953499486716e-16)(91, 4.915953499486716e-17)(92, 4.321717362186124e-18)(93, 3.288263210359007e-19)(94, 2.1214601357154885e-20)(95, 1.1284362424018556e-21)(96, 4.7513104943236025e-23)(97, 1.4847845294761258e-24)(98, 3.0614114009817026e-26)(99, 3.1238891846752067e-28)
    };

\addplot[
    color=green!50!black,
    mark='-',
]
    coordinates {
(2, 96.0)(3, 94.0)(4, 92.0)(5, 90.0)(6, 88.0)(7, 86.0)(8, 84.0)(9, 82.0)(10, 80.0)(11, 78.0)(12, 76.0)(13, 74.00000000000001)(14, 72.00000000000003)(15, 70.0000000000002)(16, 68.00000000000117)(17, 66.00000000000637)(18, 64.00000000003257)(19, 62.000000000155524)(20, 60.00000000069608)(21, 58.00000000292789)(22, 56.00000001160263)(23, 54.00000004341329)(24, 52.000000153684894)(25, 50.00000051568543)(26, 48.00000164294183)(27, 46.00000497766713)(28, 44.00001436249236)(29, 42.0000395207924)(30, 40.00010384082298)(31, 38.00026084361589)(32, 36.00062713006566)(33, 34.00144468440051)(34, 32.00319212407434)(35, 30.00677198485731)(36, 28.01380726830325)(37, 26.027081509335098)(38, 24.051147460785817)(39, 22.093104932141532)(40, 20.163505332576936)(41, 18.277281199858905)(42, 16.454533360087034)(43, 14.720954598501464)(44, 13.107650407492748)(45, 11.65015645564042)(46, 10.386559690293812)(47, 9.354796517508394)(48, 8.589395344686901)(49, 8.118102213492234)(50, 7.958923738717876)(51, 8.118102213492234)(52, 8.589395344686901)(53, 9.354796517508394)(54, 10.386559690293812)(55, 11.65015645564042)(56, 13.107650407492748)(57, 14.720954598501464)(58, 16.454533360087034)(59, 18.277281199858905)(60, 20.163505332576936)(61, 22.093104932141532)(62, 24.051147460785817)(63, 26.027081509335098)(64, 28.01380726830325)(65, 30.00677198485731)(66, 32.00319212407434)(67, 34.00144468440051)(68, 36.00062713006566)(69, 38.00026084361589)(70, 40.00010384082298)(71, 42.0000395207924)(72, 44.00001436249236)(73, 46.00000497766713)(74, 48.00000164294183)(75, 50.00000051568543)(76, 52.000000153684894)(77, 54.00000004341329)(78, 56.00000001160263)(79, 58.00000000292789)(80, 60.00000000069608)(81, 62.000000000155524)(82, 64.00000000003257)(83, 66.00000000000637)(84, 68.00000000000117)(85, 70.0000000000002)(86, 72.00000000000003)(87, 74.00000000000001)(88, 76.0)(89, 78.0)(90, 80.0)(91, 82.0)(92, 84.0)(93, 86.0)(94, 88.0)(95, 90.0)(96, 92.0)(97, 94.0)(98, 96.0)(99, 98.0)
    };
 \legend{First term of~\Cref{eq:main-K-2}, Second term of~\Cref{eq:main-K-2},
 {\(\mathbb{E}[u(n_1,n_2)]\)} } 
\end{axis}
    
        \end{tikzpicture}
    \caption{Illustration of the terms of~\Cref{eq:main-K-2} for $N=100$. The global minimum is achieved at $n_1 = 50$ as expected.}
    \label{fig:algebraic}
\end{figure}
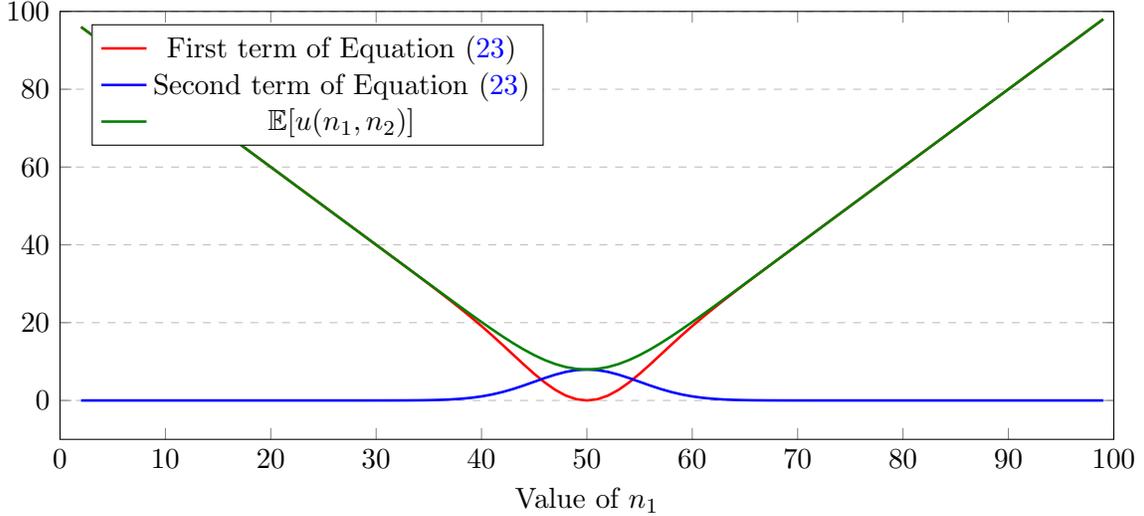

This is illustrated in~\Cref{fig:algebraic}. Therefore, in order to show that $n_1 = n_2 = N/2$ minimizes the entire expression we would need a calculation showing that the first term is \emph{more important} than the second one. At the moment we do not know of a simple algebraic way of showing that this is the case, and thus this entire proof attempt will be left at this roadblock.\footnote{We explicitly suppress the standard \emph{q.e.d} symbol.} \phantom\qedhere
\end{proof}

\section{Concluding Remarks}\label{sec:conclusion}
We have studied~\Cref{conjecture:main}, a natural conjecture regarding a very simple model of assortment optimization. Despite our best efforts, we have not succeeded in fully proving the conjecture. Nonetheless, we have provided several results representing partial steps toward a full proof. \Cref{thm:K-2} proved the conjecture for $K=2$ and~\Cref{thm:K-3} proved the conjecture for $K=3$. Both proofs are based on~\Cref{idea:max-to-min}, and proceed by induction. A full proof by induction seems challenging, as explained in~\Cref{sec:k-2}. We believe the proof of~\Cref{thm:K-3} could be expanded to $K=4$ with significant effort by designing an even more sophisticated induction scheme, but we currently do not know how such an approach could be generalized enough to proof~\Cref{conjecture:main} for all values of $K$ at once. 
Then, in~\Cref{sec:asymptotic}, we have presented asymptotic bounds for the time at which a type of goodies runs out for the first time, which can be used to provide bounds for $\mathbb{E}[u]$. We evaluated our bounds computationally in~\Cref{sec:experimental}.

On the other hand,~\Cref{sec:algebraic} presented an inconclusive algebraic approach for proving the conjecture for $K=2$. With significant more effort, similar but more complicated equations can be derived for $K=3$, but once again we do not know how such an approach could be extended for general $K$.
We hope a future reader of this article might put~\Cref{conjecture:main} to an end, and in order to motivate progress, we offer a reward of $100$ US dollars to the first person or group that proves or refutes the conjecture, in the style of greater mathematicians like Paul Erd\H{o}s, Donald Knuth or Ronald Graham.

\bibliography{main}
\bibliographystyle{plain}

\appendix
\section{Deferred Proofs}
\label{app:proofs}

\begin{proof}[Proof of~\Cref{lemma:wald}]
	Wald's lemma establishes that if $X_n$ is a countable sequence of random variables over $\mathbb{R}$ with finite mean, and $\tau$ is a random variable taking values over $\mathbb{N}$, where for every natural $n$ it holds that
	\[
		\mathbb{E}_{X_n, \tau}\left[X_n \mathbf{1}_{\tau > n}\right] = \mathbb{E}[X_n]\cdot \Pr_{\tau}[\tau > n],
	\]
	and  the infinite series satisfies
	\[
		\sum_{n=1}^\infty \mathbb{E}_{X_n, \tau}\left[|X_n|  \cdot \mathbf{1}_{\tau > n} \right]	< \infty,
	\]
	then 
	\[
		\mathbb{E}_{X_n, \tau}\left[\sum_{n=1}^\tau X_n\right] = \mathbb{E}_{\tau}\left[\sum_{n=1}^\tau \mathbb{E}[X_n]\right].
	\]
	In our case, let us define $K$ different sequences of random variables $X^{(1)}, \ldots, X^{(K)}$, where each $X^{(c)}$ is sequence of i.i.d Bernoulli distributions of parameter $p = 1/K$, corresponding to whether a goodie of color $c$ gets taken, assuming there is an infinite supply of every color.
	It is clear that each $X^{(c)}_n$ has mean $1/K$. Now let $\tau$ be exactly the smallest $j$ such that for some $c$ we have
			\[
				\sum_{n=1}^j X^{(c)}_n = n_c,
			\]
	Then, note that for every $c \in [K]$ we have
	\[
	\mathbb{E}_{X_n, \tau}\left[X^{(c)}_n \mathbf{1}_{\tau> n}\right] = \mathbb{E}\left[X^{(c)}_n\right]\cdot \Pr_{\tau}[\tau > n],
	\]
	even though $\mathbf{1}_{\tau > n}$ depends on the $X^{(c)}_n$ variables, as by total probabilities we have
	\[
	\mathbb{E}_{X_n, \tau}\left[X^{(c)}_n \mathbf{1}_{\tau > n}\right] = \mathbb{E}_{X_n}\left[X^{(c)}_n \mathbf{1}_{\tau> n} \mid \mathbf{1}_{\tau > n} \right] \Pr_{\tau}[\tau > n] +  \mathbb{E}_{X_n}\left[X^{(c)}_n \mathbf{1}_{\tau > n} \mid \mathbf{1}_{\tau \leq n} \right] \Pr_{\tau}[\tau \leq n],
	\] 
	but 
	\[
	\mathbb{E}_{X_n}\left[X^{(c)}_n \mathbf{1}_{\tau > n} \mid \mathbf{1}_{\tau > n} \right] = \mathbb{E}_{X_n}\left[X^{(c)}_n\right] \quad \text{ and } \quad 
	\mathbb{E}_{X_n}\left[X^{(c)}_n \mathbf{1}_{\tau > n} \mid \mathbf{1}_{\tau \leq n} \right] = 0,
	\]
	from where the condition is satisfied. Given that $\tau \leq N \leq n_1 + \ldots + n_K$, we trivially have that
	\[
		\sum_{n=1}^\infty \mathbb{E}_{X_n, \tau}\left[|X^{(c)}_n|  \cdot \mathbf{1}_{\tau > n} \right]	< \infty,
	\]
	from where we obtain that 
	\[
		\sum_{c=1}^K \mathbb{E}_{X_n, \tau}\left[\sum_{n=1}^\tau X^{(c)}_n\right] =  \sum_{c=1}^K\mathbb{E}_{\tau}\left[\sum_{n=1}^\tau \mathbb{E}\left[X^{(c)}_n\right]\right] = \sum_{c=1}^K\mathbb{E}_{\tau}\left[\tau/K\right] = \mathbb{E}_{\tau}[\tau] = g(n_1, \ldots, n_K).
	\]
	Let $S_c = \sum_{n=1}^\tau X^{(c)}_n$ be the total number of goodies of color $c$ taken right after the first pile gets emptied (that is, within the first $\tau$ rounds). Then, given that $S_c \leq n_c$, we have
	\begin{align*}
	E_{X_n, \tau}\left[S_c\right] &= \mathbb{E}_{X_n, \tau}\left[\sum_{n=1}^\tau X^{(c)}_n\right]\\
								&= \mathbb{E}_{\tau}\left[\sum_{n=1}^\tau \mathbb{E}\left[X^{(c)}_n\right]\right]\tag{Wald's equation}\\
								&=  \mathbb{E}_{\tau}\left[\tau/K\right]\\
								&= \mathbb{E}[\tau(n_1, \ldots, n_K)]/K.
	\end{align*}
	and given that $n_c \geq S_c$, we have
	\[
		n_c \geq E_{X_n, \tau}\left[S_c\right] = 	\mathbb{E}[\tau(n_1, \ldots, n_K)]/K.
	\]
	As we have $\mathbb{E}[\tau(n_1, \ldots, n_K)] \leq K \cdot n_c$ for every $c$, we indeed conclude that 
	\[
		\mathbb{E}[\tau(n_1, \ldots, n_K)] \leq K \cdot \min_c n_c. \qedhere
	\]
\end{proof}

The next two lemmas prove~\Cref{lemma:prob-inequality}. In particular, the proof is by induction over $n_i$, which we identify with $n_1$ while identifying $n_j$ with $n_2$, without loss of generalit. Thus, the base case $n_1 = 1$ is proven in~\Cref{lemma:1vn2}, right below, and its extension to the general case is proven in~\Cref{lemma:n1vn2}.

\begin{lemma}(1 versus $n_2$)\label{lemma:1vn2}
    For any $K>2$ and positive integers $n_3, \dots, n_K$. Let $n_1 = 1$, $n_2>1$, $p$ the random variable defined as the index in $\{1, \ldots, K\}$ of the pile which is first emptied and $\tau$ the random variable defined as the amount of rounds until the first pile gets emptied. Then,
\[
\Pr[p = 1 \mid \tau = t]\geq \Pr[p = 2 \mid \tau = t]
\]
for all $t$.
\end{lemma}

\begin{proof}
    Note that if $t < n_2$ then $\Pr[p = 2 \mid \tau = t] = 0$, and hence the lemma holds. We therefore focus on $t \geq n_2$ from now on. Let $h_i(m_1,\dots,m_j)$ denote the number of ways of coloring $i$ objects using $j$ colors such that the amount of objects getting color $c$ is less than $m_c$. Note now that we can use the $h$ function to understand our problem as follows: each round until $t$ corresponds to an object, and the color $c(t)$ of round $t$ corresponds to which pile gets an item removed in that round.
This way \(
h_t(n_1, \ldots, n_K)
\) is the number of possible sequences for the first $t$ rounds. But how many of those have $p = i$ and  $\tau = t$? To answer this, let $\overline{n_i}$  denote the sequence $(n_1, \ldots, n_{i-1}, n_{i+1}, \ldots, n_K)$, or in other words, all the pile numbers except for $n_i$. Then, there are exactly
\[
c_i = h_{t-n_i}(\overline{n_i})\cdot{t-1\choose n_i-1}
\]
such sequences, we need to choose where the $n_i-1$ will occur within the first $t-1$ rounds (note that the $t$-th round removes from pile $i$ by definition), and make sure through the $h$ function that no other pile has emptied.
Consequently, the total number of sequences with $\tau = t$ is simply
\[
C = \sum_{i=1}^K c_i.
\]


 Therefore, we have
    \begin{equation}
    \Pr[p = 1 \mid \tau = t] = \frac{c_1}{C} = \frac{h_{t-1}(n_2,n_3,\dots,n_K)}{C}
    \label{eq:p1t}
        \end{equation}
    \begin{equation}
    \Pr[p = 2 \mid \tau = t] = \frac{c_2}{C} = h_{t-n_2}(n_3,\dots,n_K)\cdot{t-1\choose n_2-1} \cdot \frac{1}{C}.
    \label{eq:p2t}
     \end{equation}

    By the definition of $h_i$ it follows that
    \begin{equation}
        h_{t-1}(n_2,n_3,\dots,n_K) = \sum_{r=0}^{n_2-1}h_{t-1-r}(n_3,\dots,n_K)\cdot{t-1\choose r}.
        \label{hiLemFer}
    \end{equation}
    Note that the term with $r=n_2-1$ in the RHS of the sum is exactly $c_2$, and as $h_i$ is always non-negative, we can conclude that $\Pr[p = 1 \mid \tau = t]\geq \Pr[p = 2 \mid \tau = t]$ for all $t$.
\end{proof}

In order to analyze the general case of comparing piles $n_1$ and $n_2$ (i.e., $n_1$ not necessarily equal to $1$) we will use a relationship between $h_i(m_1 + x,\dots,m_j)$ and $h_i(m_1,\dots,m_j)$. Indeed, note that Equation \ref{hiLemFer} does not depend on the values of $t, n_2, \dots, n_K$, then in general it leads to
\begin{align*}
    h_i(m_1 + x,m_2,\dots,m_j)&= \sum_{r=0}^{m_1 + x -1}h_{i-r}(m_2,\dots,m_j)\cdot{i\choose r}\\
    &=\sum_{r=0}^{m_1-1}h_{i-r}(m_2,\dots,m_j)\cdot{i\choose r} + \sum_{r=0}^{x-1}h_{i-m_1-r}(m_2,\dots,m_j)\cdot{i\choose m_1+r},
\end{align*}
for all $x,m_1,\dots,m_j$. By rewriting the first term of the RHS according to Equation~\ref{hiLemFer}, we obtain
\begin{equation}
h_i(m_1 + x,m_2,\dots,m_j) = h_i(m_1,m_2,\dots,m_j) + \sum_{r=0}^{x-1}h_{i-m_1-r}(m_2,\dots,m_j)\cdot{i\choose m_1+r}.
    \label{i-xLemaFer}
\end{equation}

\begin{lemma}($n_1$ versus $n_2$)~\label{lemma:n1vn2}
	 The previous lemma holds for all positive numbers $n_1, n_2$ such that $n_1 < n_2$.
\end{lemma}

\begin{proof}
For every positive number $d:=n_2-n_1$, the proof is by induction on $n_1$. The previous lemma makes the base case for all $d$, therefore we now argue that if the lemma holds for $(n_1, n_2)$, with $n_2 - n_1 = d$, then it must hold for $(n_1 + 1, n_2 +1)$. Our inductive hypothesis is therefore that $\Pr[p = 1 \mid \tau = t] \geq \Pr[p = 2 \mid \tau = t]$ for all $t$, which by Equations~\ref{eq:p1t} and~\ref{eq:p2t}, is equivalent to 
\begin{equation}
    h_{t-n_1}(n_2, n_3,\dots,n_K)\cdot{t-1\choose n_1-1}\geq h_{t-n_2}(n_1, n_3,\dots,n_K)\cdot{t-1\choose n_2-1},
    \label{eq:inductive_h}
\end{equation}
for all values of $t$.
We now proceed to show that for any $t$ it holds that
\begin{equation*}
    h_{t-(n_1+1)}(n_2+1, n_3,\dots,n_K)\cdot{t-1\choose n_1}\geq h_{t-(n_2+1)}(n_1+1, n_3,\dots,n_K)\cdot{t-1\choose n_2}.
\end{equation*}
In order to improve legibility, we abbreviate ``$n_3,\dots,n_K$'' as $n_3^K$. Using the Equation \ref{i-xLemaFer} with $x=1$, the inequality to be proven is equivalent to
\begin{multline}
    \left[h_{t-n_1-1}(n_2, n_3^K) + h_{t-n_1-1-n_2}(n_3^K)\cdot{t-n_1\choose n_2}\right]\cdot{t-1\choose n_1}   \\
    \geq \left[h_{t-n_2-1}(n_1, n_3^K) + h_{t-n_2-1-n_1}(n_3^K)\cdot{t-n_2\choose n_1}\right]\cdot{t-1\choose n_2}.
    \label{hypLemmaFer}
\end{multline}

If we use our inductive hypothesis of Equation~\ref{eq:inductive_h}  for $t-1$, we get

\[
    h_{t-1-n_1}(n_2, n_3^K) \cdot {t-2 \choose n_1 - 1} \geq
    h_{t-1-n_2}(n_1, n_3^K) \cdot {t-2 \choose n_2 - 1}.
\]

Now, we use this as follows:

\begin{align}
h_{t-n_1-1} (n_2, n_3^K) \cdot {t-1 \choose n_1 }
    &=
    h_{t-n_1-1}(n_2,n_3^K)\cdot{t-2\choose n_1-1}\cdot\frac{t-1}{n_1} \notag\\
    &\geq h_{t-n_2-1}(n_1,n_3^K)\cdot{t-2\choose n_2-1} \cdot \frac{t-1}{n_2}\tag{Inductive Hyp.,  $n_1 < n_2$}\\
    & = h_{t-n_2-1} (n_1, n_3^K) \cdot {t-1 \choose n_2 }    \label{LemFeri}
\end{align}

From Equation~\ref{LemFeri}, what remains to be prove of the desired Equation~\ref{hypLemmaFer}
is
\[
 h_{t-n_1-1-n_2}(n_3^K) \cdot {t-n_1 \choose n_2} \cdot {t-1 \choose n_1} \geq h_{t- n_2 - 1- n_1}(n_3^K) \cdot {t-n_2 \choose n_1} \cdot {t-1 \choose n_2},
\]
or equivalently,
\[
{t-n_1 \choose n_2} \cdot {t-1 \choose n_1} \geq {t-n_2 \choose n_1} \cdot {t-1 \choose n_2}.
\]
But

\[\left[{t-n_1\choose n_2} \cdot {t-1 \choose n_1}\right] \div \left[{t-n_2\choose n_1}\cdot{t-1 \choose n_2}\right] = \frac{t-n_1}{t-n_2},\]
 which is greater than $1$, and thus our lemma is finally proven. 
\end{proof}

\section{Figures}
\label{app:figures}

\begin{figure}
\begin{tikzpicture}
\begin{axis}[
every axis plot/.append style={line width=0.975pt},
    title={},
    width=\linewidth,
    height=2*\axisdefaultheight,
    xlabel={Sample index under lexicographic sorting.},
    ylabel={N$^{\circ}$ of rounds },
    xmin=0,
    xmax=30,
    ymin=0,
    ymajorgrids=true,
    grid style=dashed,
    legend style={anchor=north, legend columns=2, font=\small
    },
    legend pos=north west,
]

\addplot[
    color=black,
    mark=o,
    ]
    coordinates{
    (1,9.934)(2,8.4393)(3,9.7213)(4,19.953)(5,29.9434)(6,29.9405)(7,30.2424)(8,34.8575)(9,35.0919)(10,40.0434)(11,39.9822)(12,44.87)(13,45.0767)(14,41.5046)(15,44.4105)(16,54.7736)(17,70.2762)(18,75.0147)(19,84.078)(20,89.4784)(21,98.0616)(22,103.6288)(23,106.4502)(24,120.5994)(25,127.3087)(26,147.2139)(27,146.3782)(28,168.7921)(29,185.8892)(30,195.5127)
    };
    \addlegendentry{$\hat{\tau},~$}

\addplot[
    color=black!75!white,
    mark='-',
    line width=1pt,
    dashed
    ]
    coordinates{
    (1,10)(2,10)(3,10)(4,20)(5,30)(6,30)(7,30)(8,35)(9,35)(10,40)(11,40)(12,45)(13,45)(14,45)(15,45)(16,55)(17,70)(18,75)(19,85)(20,90)(21,100)(22,105)(23,120)(24,125)(25,140)(26,150)(27,160)(28,170)(29,190)(30,200)
    };
    \addlegendentry{$K \cdot \min_c n_c, ~~~\mathcal{U}(5,[1,100])$}

\addplot[
    color=brown,
    mark=*,
    ]
    coordinates{
    (1,9.9419)(2,15.0631)(3,29.928)(4,40.0772)(5,35.7298)(6,44.9667)(7,49.6117)(8,49.53)(9,54.8963)(10,74.7897)(11,74.8293)(12,79.8677)(13,94.9931)(14,94.9597)(15,99.8048)(16,103.339)(17,122.7884)(18,115.6039)(19,134.8273)(20,139.6631)(21,149.54)(22,220.1196)(23,244.3535)(24,276.4216)(25,264.6609)(26,284.7899)(27,290.3078)(28,285.8643)(29,324.6416)(30,374.6063)
    };
    \addlegendentry{$\hat{\tau},~$}

\addplot[
    color=brown!75!white,
    mark='-',
    line width=1pt,
    dashed
    ]
    coordinates{
    (1,10)(2,15)(3,30)(4,40)(5,40)(6,45)(7,50)(8,50)(9,55)(10,75)(11,75)(12,80)(13,95)(14,95)(15,100)(16,105)(17,125)(18,130)(19,135)(20,140)(21,150)(22,220)(23,245)(24,280)(25,280)(26,285)(27,290)(28,310)(29,325)(30,375)
    };
    \addlegendentry{$K \cdot \min_c n_c, ~~~\mathcal{U}(5,[1,150])$}

\addplot[
    color=red!80!black,
    mark=triangle,
    ]
    coordinates{
    (1,4.9415)(2,4.9469)(3,24.9603)(4,30.1307)(5,30.1822)(6,35.1324)(7,45.2338)(8,44.9019)(9,50.0285)(10,80.246)(11,94.8009)(12,103.7654)(13,105.0236)(14,105.1129)(15,120.0322)(16,159.9623)(17,179.4645)(18,188.5762)(19,203.5655)(20,224.9599)(21,230.7413)(22,250.3201)(23,265.9379)(24,280.1344)(25,325.9412)(26,406.4677)(27,428.452)(28,429.3769)(29,469.8347)(30,482.1521)
    };
    \addlegendentry{$\hat{\tau},~$}

\addplot[
    color=red!50!white,
    mark='-',
    line width=1pt,
    dashed
    ]
    coordinates{
    (1,5)(2,5)(3,25)(4,30)(5,30)(6,35)(7,45)(8,45)(9,50)(10,80)(11,95)(12,105)(13,105)(14,105)(15,120)(16,160)(17,180)(18,190)(19,220)(20,225)(21,235)(22,250)(23,270)(24,295)(25,345)(26,410)(27,435)(28,440)(29,485)(30,495)
    };
    \addlegendentry{$K \cdot \min_c n_c, ~~~\mathcal{U}(5,[1,200])$}

\addplot[
    color=green!50!black,
    mark=diamond,
    ]
    coordinates{
    (1,10.037)(2,10.0036)(3,9.7693)(4,9.0274)(5,9.6171)(6,10.0195)(7,9.2534)(8,9.378)(9,15.63)(10,29.4404)(11,30.04)(12,26.1642)(13,28.4717)(14,28.9556)(15,29.5194)(16,23.9713)(17,29.1268)(18,39.2412)(19,35.2513)(20,33.8035)(21,36.949)(22,44.0554)(23,53.1577)(24,68.8683)(25,69.4012)(26,65.8378)(27,89.5372)(28,90.2103)(29,103.9039)(30,110.5215)
    };
    \addlegendentry{$\hat{\tau},~$}

\addplot[
    color=green!50!white,
    mark='-',
    line width=1pt,
    dashed
    ]
    coordinates{
    (1,10)(2,10)(3,10)(4,10)(5,10)(6,10)(7,10)(8,20)(9,20)(10,30)(11,30)(12,30)(13,30)(14,30)(15,30)(16,30)(17,30)(18,40)(19,40)(20,40)(21,40)(22,50)(23,60)(24,70)(25,70)(26,80)(27,100)(28,100)(29,120)(30,120)
    };
    \addlegendentry{$K \cdot \min_c n_c, ~~~\mathcal{U}(10,[1,50])$}

\addplot[
    color=blue!80!black, 
    mark=square,
    ]
    coordinates{
    (1,9.1749)(2,4.879)(3,8.7873)(4,9.9205)(5,19.5763)(6,19.8129)(7,29.5284)(8,28.9249)(9,39.9907)(10,35.169)(11,39.9602)(12,39.9108)(13,29.141)(14,50.0785)(15,46.0334)(16,41.647)(17,57.6858)(18,58.7554)(19,54.9278)(20,70.6439)(21,64.2634)(22,79.8152)(23,64.0309)(24,88.8161)(25,116.4721)(26,136.3903)(27,124.4048)(28,159.6326)(29,172.1167)(30,190.4695)
    };
    \addlegendentry{$\hat{\tau},~$}

\addplot[
    color=blue!50!white,
    mark='-',
    line width=1pt,
    dashed
    ]
    coordinates{
    (1,10)(2,10)(3,10)(4,10)(5,20)(6,20)(7,30)(8,30)(9,40)(10,40)(11,40)(12,40)(13,40)(14,50)(15,50)(16,50)(17,60)(18,60)(19,60)(20,70)(21,80)(22,80)(23,80)(24,90)(25,120)(26,140)(27,140)(28,160)(29,180)(30,190)
    };
    \addlegendentry{$K \cdot \min_c n_c, ~~~\mathcal{U}(10,[1,80])$}

\addplot[
    color=purple!75!black,
    mark=+,
    ]
    coordinates{
    (1,9.9158)(2,9.8035)(3,9.736)(4,15.5321)(5,17.5608)(6,20.2144)(7,12.2969)(8,29.7609)(9,30.1136)(10,39.9526)(11,39.9962)(12,36.9035)(13,46.9319)(14,50.0927)(15,49.6191)(16,59.8711)(17,64.7855)(18,62.1687)(19,89.9929)(20,89.9396)(21,83.2425)(22,86.095)(23,110.1173)(24,129.1453)(25,140.2527)(26,133.337)(27,147.6135)(28,169.586)(29,173.9076)(30,320.8101)
    };
    \addlegendentry{$\hat{\tau},~$}

\addplot[
    color=purple!50!white,
    mark='-',
    dashed
    ]
    coordinates{
    (1,10)(2,10)(3,10)(4,20)(5,20)(6,20)(7,20)(8,30)(9,30)(10,40)(11,40)(12,40)(13,50)(14,50)(15,50)(16,60)(17,70)(18,70)(19,90)(20,90)(21,100)(22,100)(23,110)(24,130)(25,140)(26,140)(27,150)(28,170)(29,190)(30,350)
    };
    \addlegendentry{$K \cdot \min_c n_c, ~~~\mathcal{U}(10,[1,100])$}

\end{axis}

\end{tikzpicture}
  \caption{Measured $\hat{\tau}$ on 30 random samples from  $\mathcal{U}^5([1,100])$, $\mathcal{U}^5([1,150])$, $\mathcal{U}^5([1,200])$, $\mathcal{U}^{10}([1,50])$, $\mathcal{U}^{10}([1,80])$, $\mathcal{U}^{10}([1,100])$.}
 \label{fig:approx-tau}
\end{figure}
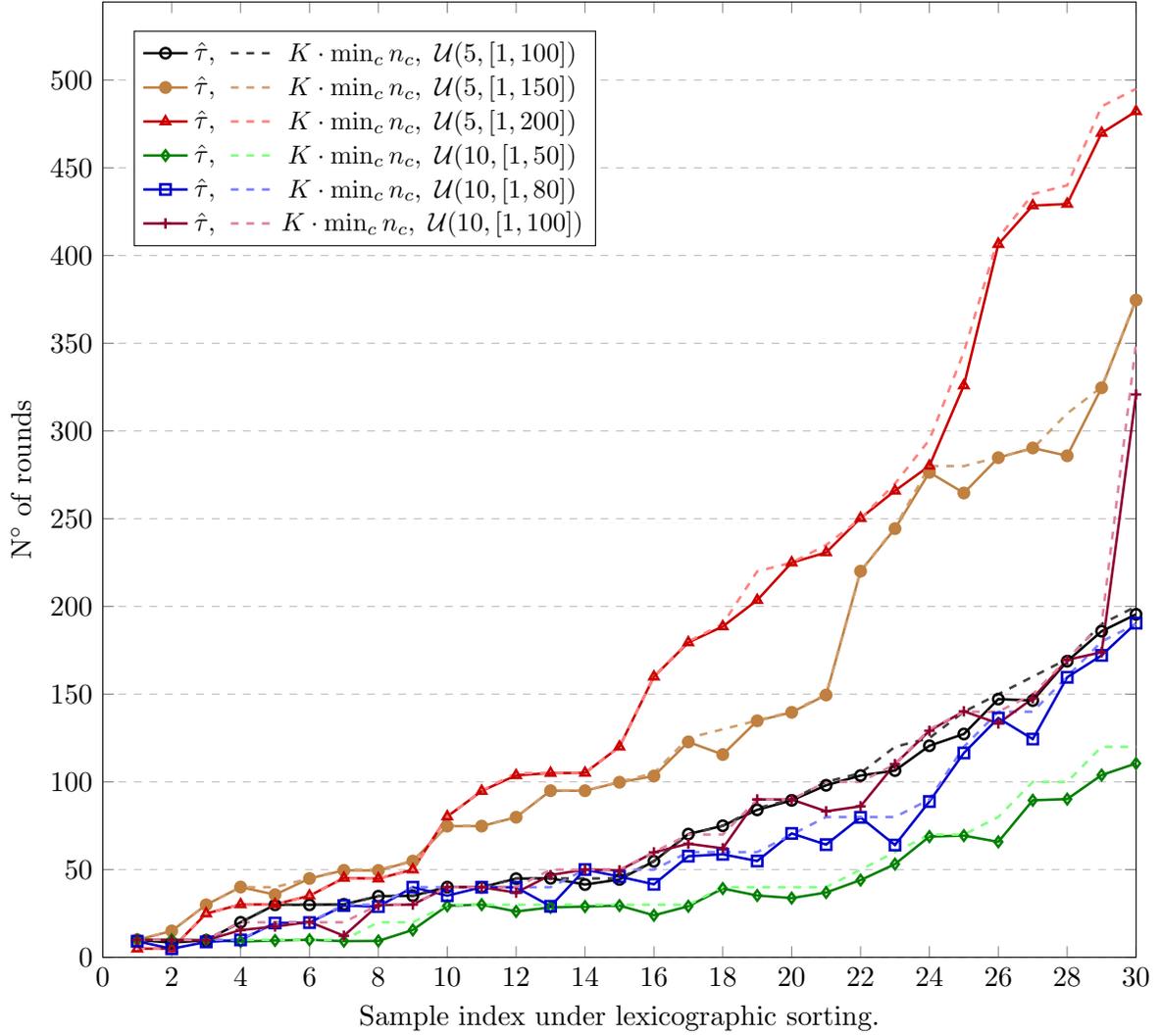

\begin{figure}
\begin{tikzpicture}
\begin{axis}[
every axis plot/.append style={line width=0.975pt},
    title={},
    width=\linewidth,
    height=2*\axisdefaultheight,
    xlabel={Sample index under lexicographic sorting.},
    ylabel={N$^{\circ}$ of rounds },
    xmin=0,
    xmax=30,
    ymin=0,
    ymax=320,
    ymajorgrids=true,
    grid style=dashed,
    legend style={anchor=north, legend columns=2, font=\small
    },
    legend pos=north west,
]

\addplot[
    color=black,
    mark=o,
    ]
    coordinates{
    (1,5.0331)(2,5.0213)(3,14.868)(4,14.7812)(5,13.7966)(6,14.9382)(7,14.9995)(8,19.5138)(9,30.0793)(10,29.8787)(11,34.6533)(12,34.4646)(13,40.0454)(14,40.0461)(15,45.1088)(16,42.4086)(17,50.245)(18,74.7504)(19,83.6878)(20,83.1132)(21,82.7851)(22,88.4792)(23,86.7434)(24,105.1201)(25,105.0324)(26,100.56)(27,118.8286)(28,137.3765)(29,133.3235)(30,161.0493)};
    \addlegendentry{$\hat{\gamma},~$}

\addplot[
    color=black!75!white,
    mark='-',
    dashed
    ]
    coordinates{
    (1,5)(2,5)(3,15)(4,15)(5,15)(6,15)(7,15)(8,20)(9,30)(10,30)(11,35)(12,35)(13,40)(14,40)(15,45)(16,45)(17,50)(18,75)(19,85)(20,85)(21,85)(22,95)(23,95)(24,105)(25,105)(26,110)(27,120)(28,140)(29,145)(30,180)
    };
    \addlegendentry{$K \cdot \min_c n_c, ~~~\mathcal{S}(5,250)$}

\addplot[
    color=brown,
    mark=*,
    ]
    coordinates{
    (1,5.0359)(2,19.9868)(3,25.096)(4,24.4915)(5,24.9627)(6,30.0736)(7,33.1984)(8,35.0566)(9,50.0248)(10,49.6527)(11,55.1544)(12,54.1888)(13,59.8597)(14,69.9325)(15,70.0921)(16,79.9185)(17,72.7885)(18,80.2563)(19,99.8943)(20,99.3682)(21,114.7687)(22,117.7758)(23,114.6278)(24,124.9281)(25,122.6437)(26,138.4924)(27,163.8124)(28,173.5975)(29,171.184)(30,183.7385)
    };
    \addlegendentry{$\hat{\gamma},~$}

\addplot[
    color=brown!75!white,
    mark='-',
    dashed
    ]
    coordinates{
    (1,5)(2,20)(3,25)(4,25)(5,25)(6,30)(7,35)(8,35)(9,50)(10,50)(11,55)(12,60)(13,60)(14,70)(15,70)(16,80)(17,80)(18,90)(19,100)(20,100)(21,115)(22,120)(23,120)(24,125)(25,125)(26,140)(27,170)(28,175)(29,175)(30,185)
    };
    \addlegendentry{$K \cdot \min_c n_c, ~~~\mathcal{S}(5,375)$}

\addplot[
    color=red!80!black,
    mark=triangle,
    ]
    coordinates{
    (1,5.0183)(2,4.8774)(3,10.0292)(4,19.8057)(5,20.0841)(6,34.7992)(7,45.1567)(8,49.9766)(9,50.221)(10,55.0317)(11,69.9234)(12,70.0434)(13,100.3621)(14,125.0292)(15,128.5195)(16,130.3446)(17,124.3487)(18,134.968)(19,135.002)(20,134.8929)(21,144.9102)(22,145.1121)(23,174.5476)(24,175.1776)(25,194.6681)(26,218.8905)(27,239.3912)(28,246.5976)(29,245.9505)(30,301.3976)
    };
    \addlegendentry{$\hat{\gamma},~$}

\addplot[
    color=red!50!white,
    mark='-',
    dashed
    ]
    coordinates{
    (1,5)(2,5)(3,10)(4,20)(5,20)(6,35)(7,45)(8,50)(9,50)(10,55)(11,70)(12,70)(13,100)(14,125)(15,130)(16,130)(17,130)(18,135)(19,135)(20,135)(21,145)(22,145)(23,175)(24,175)(25,195)(26,220)(27,240)(28,250)(29,260)(30,315)
    };
    \addlegendentry{$K \cdot \min_c n_c, ~~~\mathcal{S}(5,500)$}

\addplot[
    color=green!50!black,
    mark=diamond,
    ]
    coordinates{
    (1,5.0115)(2,7.0268)(3,8.7193)(4,9.9081)(5,8.9963)(6,9.8252)(7,19.7039)(8,19.9363)(9,20.1825)(10,27.9124)(11,28.6762)(12,27.5474)(13,29.8219)(14,29.3068)(15,25.808)(16,29.7291)(17,39.0243)(18,33.8105)(19,29.8766)(20,32.7152)(21,36.7223)(22,37.3584)(23,57.8493)(24,50.516)(25,42.8722)(26,50.3445)(27,55.3539)(28,78.6321)(29,76.9802)(30,90.8498)
    };
    \addlegendentry{$\hat{\gamma},~$}

\addplot[
    color=green!50!white,
    mark='-',
    dashed
    ]
    coordinates{
    (1,10)(2,10)(3,10)(4,10)(5,10)(6,10)(7,20)(8,20)(9,20)(10,30)(11,30)(12,30)(13,30)(14,30)(15,30)(16,30)(17,40)(18,40)(19,40)(20,40)(21,40)(22,50)(23,60)(24,60)(25,60)(26,60)(27,70)(28,80)(29,100)(30,110)
    };
    \addlegendentry{$K \cdot \min_c n_c, ~~~\mathcal{S}(10,250)$}

\addplot[
    color=blue!80!black,
    mark=square,
    ]
    coordinates{
    (1,10.0643)(2,9.9312)(3,8.6308)(4,4.3691)(5,20.0152)(6,19.4001)(7,18.5548)(8,29.7179)(9,29.9384)(10,26.411)(11,29.9148)(12,28.9381)(13,29.9207)(14,36.8749)(15,39.4931)(16,38.3526)(17,38.7136)(18,50.3788)(19,47.3627)(20,59.7676)(21,68.3427)(22,53.9777)(23,65.9354)(24,82.4571)(25,105.1053)(26,105.3012)(27,90.4947)(28,119.8385)(29,126.604)(30,215.5269)
    };
    \addlegendentry{$\hat{\gamma},~$}

\addplot[
    color=blue!50!white,
    mark='-',
    dashed
    ]
    coordinates{
    (1,10)(2,10)(3,10)(4,10)(5,20)(6,20)(7,20)(8,30)(9,30)(10,30)(11,30)(12,30)(13,30)(14,40)(15,40)(16,40)(17,40)(18,50)(19,60)(20,60)(21,70)(22,70)(23,70)(24,90)(25,110)(26,110)(27,110)(28,120)(29,130)(30,250)
    };
    \addlegendentry{$K \cdot \min_c n_c, ~~~\mathcal{S}(10,400)$}

\addplot[
    color=purple!75!black,
    mark=+,
    ]
    coordinates{
    (1,8.4671)(2,10.2952)(3,19.855)(4,20.1421)(5,20.0094)(6,30.3393)(7,39.4524)(8,40.1792)(9,43.8111)(10,41.2605)(11,48.1484)(12,45.3596)(13,41.6124)(14,47.6408)(15,47.7418)(16,58.8865)(17,55.414)(18,68.7462)(19,70.0684)(20,70.0079)(21,67.8)(22,68.307)(23,64.19)(24,80.1444)(25,79.6058)(26,88.2479)(27,100.3469)(28,113.051)(29,117.5563)(30,124.9908)
    };
    \addlegendentry{$\hat{\gamma},~$}

\addplot[
    color=purple!50!white,
    mark='-',
    dashed
    ]
    coordinates{
    (1,10)(2,10)(3,20)(4,20)(5,20)(6,30)(7,40)(8,40)(9,50)(10,50)(11,50)(12,50)(13,50)(14,50)(15,50)(16,60)(17,60)(18,70)(19,70)(20,70)(21,70)(22,80)(23,80)(24,80)(25,80)(26,90)(27,110)(28,120)(29,130)(30,130)
    };
    \addlegendentry{$K \cdot \min_c n_c, ~~~\mathcal{S}(10,500)$}

\end{axis}

\end{tikzpicture}
  \caption{Measured $\hat{\tau}$ on 30 random samples from  $\mathcal{S}(5, 250)$, $\mathcal{S}(5, 375)$, $\mathcal{S}(5, 500)$, $\mathcal{S}(10, 250)$, $\mathcal{S}(10, 400)$, $\mathcal{S}(10, 500)$.}
 \label{fig:S-distribution}
\end{figure}
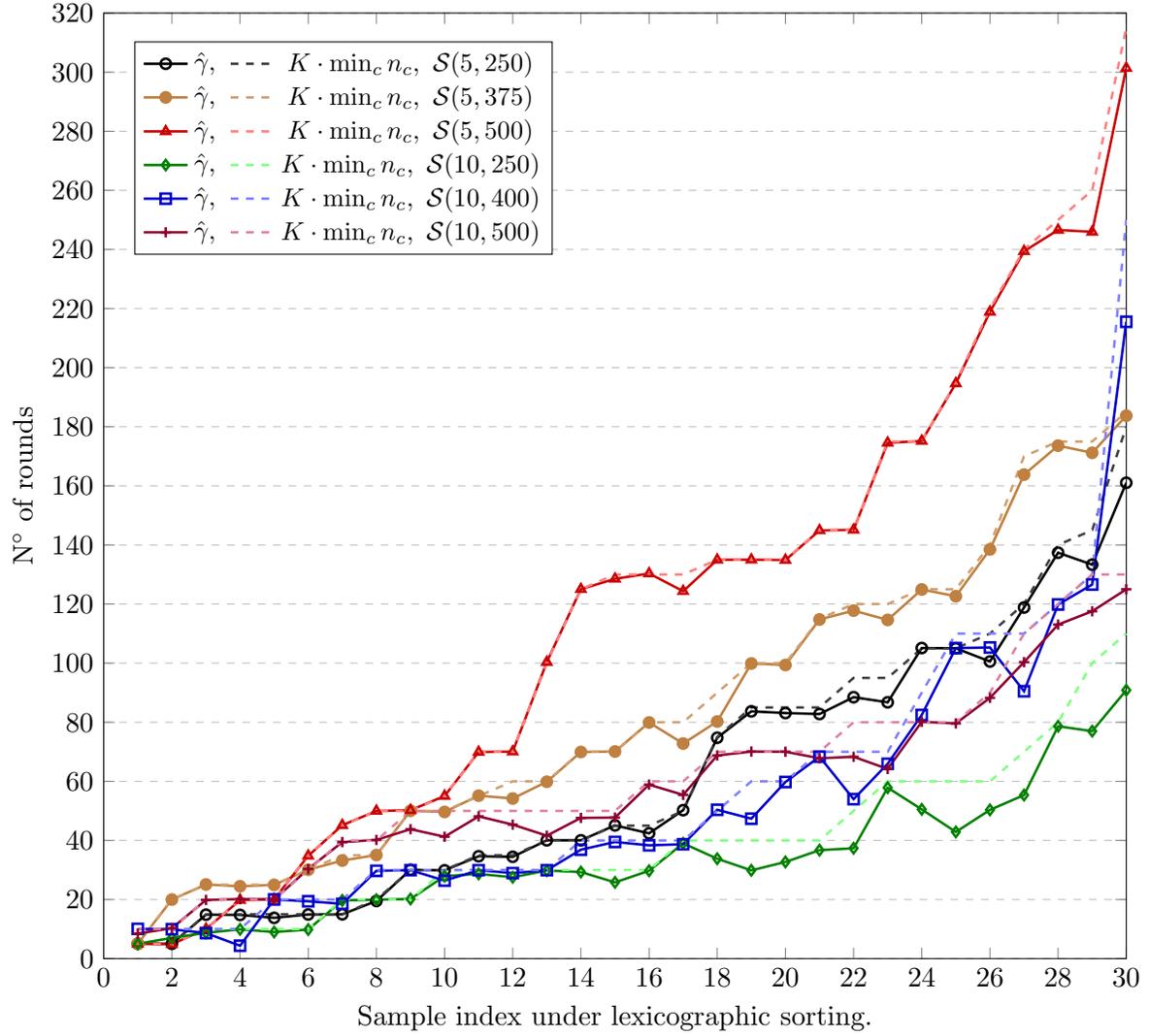

\input{fig6}

\input{fig7}

\input{fig8}

\input{fig9}

\begin{figure}
    \centering
    \begin{tikzpicture}
        \begin{axis}[
        every axis plot/.append style={line width=0.975pt},
            title={},
            width=\linewidth,
            height=\axisdefaultheight,
            xlabel={Sample index},
            ylabel={N$^{\circ}$ of rounds  },
            xmin=0,
            xmax=100,
            ymin=5000,
            ymajorgrids=true,
            grid style=dashed,
            legend pos=north west,
        ]
        \addplot[color=black, mark='-'] coordinates { (0, 9426.3)(1, 9642.65)(2, 9945.1)(3, 10041.75)(4, 9371.1)(5, 9689.15)(6, 9731.3)(7, 10066.3)(8, 10211.9)(9, 9748.1)(10, 11555.3)(11, 9682.25)(12, 10209.9)(13, 9750.4)(14, 10092.85)(15, 9752.35)(16, 9932.05)(17, 9608.2)(18, 9702.35)(19, 9533.5)(20, 9731.8)(21, 9464.45)(22, 9809.4)(23, 9506.85)(24, 9572.3)(25, 10079.7)(26, 9995.65)(27, 9581.7)(28, 10099.85)(29, 9244.25)(30, 9574.75)(31, 9870.6)(32, 11072.2)(33, 9556.4)(34, 9855.75)(35, 9791.35)(36, 9638.45)(37, 9979.65)(38, 9528.6)(39, 10185.1)(40, 10189.5)(41, 10275.8)(42, 9708.15)(43, 9518.4)(44, 9664.2)(45, 10196.05)(46, 9772.5)(47, 9874.45)(48, 9852.9)(49, 9635.7)(50, 9816.6)(51, 9652.55)(52, 10568.0)(53, 9869.25)(54, 10388.3)(55, 10552.7)(56, 9756.55)(57, 9341.3)(58, 10016.6)(59, 9705.9)(60, 9893.2)(61, 10130.25)(62, 10094.05)(63, 9914.5)(64, 10071.25)(65, 10190.75)(66, 9963.1)(67, 9923.15)(68, 9554.4)(69, 9937.8)(70, 10423.25)(71, 9368.3)(72, 10530.4)(73, 10077.85)(74, 9472.0)(75, 10566.35)(76, 9735.45)(77, 9875.9)(78, 9989.6)(79, 10666.8)(80, 9851.7)(81, 9566.5)(82, 9517.65)(83, 10037.95)(84, 10209.0)(85, 9394.55)(86, 10066.35)(87, 10387.0)(88, 9550.15)(89, 9410.7)(90, 9695.55)(91, 9608.55)(92, 9648.5)(93, 9689.55)(94, 10764.25)(95, 9648.0)(96, 9791.4)(97, 10240.0)(98, 9672.95)(99, 10195.45) };

        \addplot[color=blue, dashed, mark='-'] coordinates { (0, 6826.670800627617)(1, 6946.4470851639935)(2, 7427.249649427384)(3, 7427.249649427384)(4, 6826.670800627617)(5, 6986.411219120156)(6, 6986.411219120156)(7, 7146.457627743594)(8, 7026.394475423982)(9, 7146.457627743594)(10, 8315.209874919992)(11, 7026.394475423982)(12, 7387.082386479119)(13, 6906.502211049488)(14, 7427.249649427384)(15, 7106.41781321269)(16, 6946.4470851639935)(17, 6866.576735926264)(18, 6866.576735926264)(19, 6946.4470851639935)(20, 6826.670800627617)(21, 6906.502211049488)(22, 7146.457627743594)(23, 6866.576735926264)(24, 6906.502211049488)(25, 7106.41781321269)(26, 7026.394475423982)(27, 6866.576735926264)(28, 7266.688084697089)(29, 6866.576735926264)(30, 7026.394475423982)(31, 7186.516030662162)(32, 8072.2563819921015)(33, 6906.502211049488)(34, 7026.394475423982)(35, 6826.670800627617)(36, 6906.502211049488)(37, 7186.516030662162)(38, 6826.670800627617)(39, 7146.457627743594)(40, 7467.434624108526)(41, 7668.621033592715)(42, 6986.411219120156)(43, 6826.670800627617)(44, 6866.576735926264)(45, 7467.434624108526)(46, 7026.394475423982)(47, 7266.688084697089)(48, 7387.082386479119)(49, 6826.670800627617)(50, 7106.41781321269)(51, 6866.576735926264)(52, 7507.637190911511)(53, 7146.457627743594)(54, 7226.592892352248)(55, 7467.434624108526)(56, 7226.592892352248)(57, 6866.576735926264)(58, 7146.457627743594)(59, 6906.502211049488)(60, 7146.457627743594)(61, 7427.249649427384)(62, 7186.516030662162)(63, 7146.457627743594)(64, 6946.4470851639935)(65, 7306.801481055132)(66, 7106.41781321269)(67, 7106.41781321269)(68, 6906.502211049488)(69, 6946.4470851639935)(70, 7749.215494337081)(71, 6906.502211049488)(72, 7547.857231565537)(73, 7026.394475423982)(74, 6906.502211049488)(75, 7708.909813023698)(76, 6866.576735926264)(77, 6906.502211049488)(78, 7306.801481055132)(79, 7507.637190911511)(80, 7066.396718209659)(81, 6866.576735926264)(82, 6906.502211049488)(83, 7186.516030662162)(84, 7306.801481055132)(85, 6866.576735926264)(86, 7106.41781321269)(87, 7507.637190911511)(88, 6866.576735926264)(89, 6866.576735926264)(90, 6826.670800627617)(91, 6826.670800627617)(92, 6906.502211049488)(93, 7026.394475423982)(94, 7749.215494337081)(95, 7106.41781321269)(96, 6826.670800627617)(97, 7387.082386479119)(98, 6826.670800627617)(99, 7186.516030662162) };

\legend{$\hat{\tau}$,$(K-2)\cdot \min_c n_c - \delta$}

        \end{axis}
    \end{tikzpicture}
    \caption{Illustration for 100 samples from $\mathcal{U}^{50}([200, 500])$.}
    \label{fig:large-K}
\end{figure}
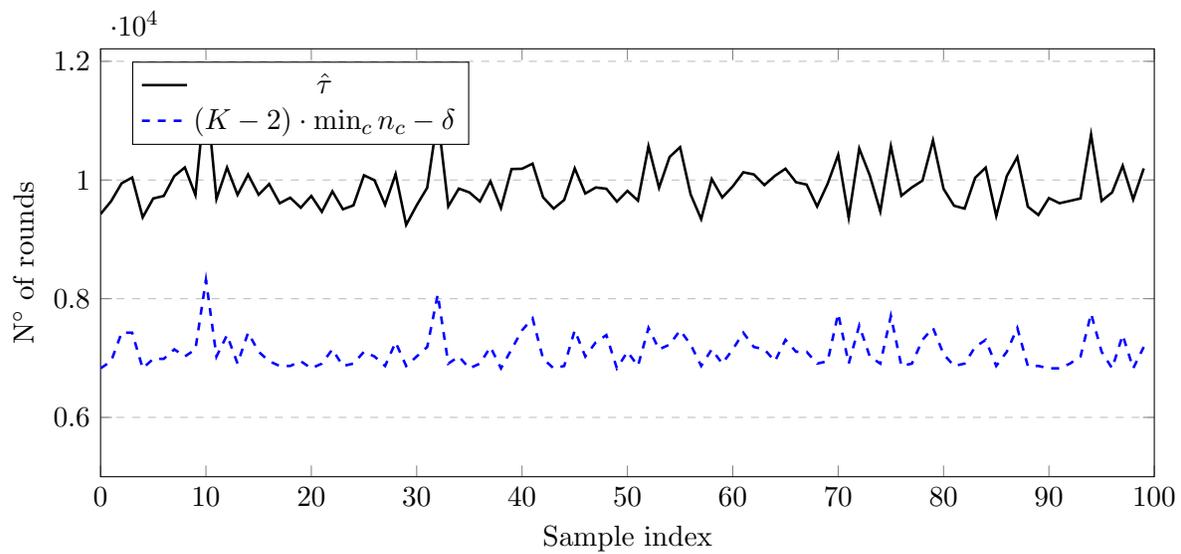

\end{document}